\documentclass[11pt, a4paper]{amsart}

\usepackage[latin1]{inputenc}
\usepackage[all]{xy}
\SelectTips{cm}{}
\usepackage[pagebackref,
  ,pdfborder={0 0 0}
  ,urlcolor=black,a4paper,hypertexnames=false]{hyperref}

\usepackage{amsfonts}
\usepackage{amssymb, amsmath,amsthm, mathtools}
\usepackage{tabularx, hyperref}
\usepackage{enumerate}

\usepackage{tikz-cd}

\makeatletter
\newtheorem*{rep@theorem}{\rep@title}
\newcommand{\newreptheorem}[2]{%
\newenvironment{rep#1}[1]{%
 \def\rep@title{#2 \ref{##1}}%
 \begin{rep@theorem}}%
 {\end{rep@theorem}}}
\makeatother

\newtheorem*{thm*}{Theorem}
\newtheorem{introthm}{Theorem}

\newtheorem*{intro_thmA}{Theorem A}
\newtheorem*{intro_thmB}{Theorem B}
\newtheorem*{intro_thmC}{Theorem C}

\theoremstyle{definition}
\newtheorem*{intro_defi*}{Definition}
\newtheorem*{intro_rem*}{Remark}

\theoremstyle{theorem}
\newtheorem{lemma}{Lemma}[subsection]
\newtheorem{thm}[lemma]{Theorem}
\newreptheorem{thm}{Theorem}  
\newtheorem{prop}[lemma]{Proposition}
\newreptheorem{prop}{Proposition}  
\newtheorem{cor}[lemma]{Corollary}
\newreptheorem{cor}{Corollary}  

\theoremstyle{definition}
\newtheorem{defi}[lemma]{Definition}
\newreptheorem{defi}{Definition}  

\newreptheorem{quest}{Question}  

\newtheorem{example}[lemma]{Example}
\newtheorem{rem}[lemma]{Remark}
\newtheorem{nota}[lemma]{Notation}

\newtheorem{scholium}[lemma]{Scholium}

\theoremstyle{definition}

\makeatletter
\newcommand\norm{\bBigg@{0.8}}

\makeatother
 
 \newcommand{\indnorm}[2][flex]{\csname #1l\endcsname\|#2%
                                 \csname #1r\endcsname\|\mathclose{}}
                                  \newcommand{\indnorml}[4][flex]{\csname #1l\endcsname\|#2%
                                 \csname #1r\endcsname\|_{#3}^{#4}\mathclose{}}
\newcommand{\sv}[2][flex]{\indnorm[#1]{#2}}

\DeclareMathOperator{\res}{res}
\DeclareMathOperator{\Res}{\textbf{res}}

\DeclareMathOperator{\comp}{comp}

\DeclareMathOperator{\vbcr}{\mathcal{BA}c}

\newcommand{\N}{\ensuremath {\mathbb{N}}}
\newcommand{\R} {\ensuremath {\mathbb{R}}}

\newcommand{\Z} {\ensuremath {\mathbb{Z}}}

\newcommand{\ModG}{\mathbf{Mod}_{\R}^{\Gamma}}
\newcommand{\Mod}{\mathbf{Csn}_{\R}}

\renewcommand{\rho}{\varrho}
\def\phi{\varphi}

\usepackage{color}
\usepackage{pdfcolmk}

\long\def\forget#1{}

\def\widetilde{\tilde}

\makeatletter
\@namedef{subjclassname@2020}{%
  \textup{2020} Mathematics Subject Classification}
\makeatother

\begin{document}

\title[Amenability and Acyclicity in Bounded Cohomology]{Amenability and Acyclicity in Bounded Cohomology Theory}

\author{Marco Moraschini}
\address{\newline M. Moraschini \newline
Dipartimento di Matematica, Universit\`{a} di Bologna, Bologna, Italia}
\email{marco.moraschini2@unibo.it}

\author{George Raptis}
\address{\newline G. Raptis \newline
Fakult\"{a}t f\"{u}r Mathematik, Universit\"{a}t Regensburg, Regensburg, Germany}
\email{georgios.raptis@ur.de}

\thanks{}


\begin{abstract}
Johnson's characterization of amenable groups states that a discrete group $\Gamma$ is amenable
if and only if $H_b^{n \geq 1}(\Gamma; V) = 0$ for all dual normed $\R[\Gamma]$-modules $V$. 
In this paper, we extend the previous result to homomorphisms by proving the converse of the
\emph{Mapping Theorem}:
a surjective group homomorphism $\phi \colon \Gamma \to K$ has amenable kernel $H$ if and only if the induced inflation map 
$H^\bullet_b(K; V^H) \to H^\bullet_b(\Gamma; V)$ is an isometric isomorphism for every dual normed $\R[\Gamma]$-module 
$V$.
In addition, we obtain an analogous characterization for the (smaller) class of surjective group homomorphisms $\phi \colon \Gamma \to K$
with the property that the inflation maps in bounded cohomology are isometric isomorphisms for \emph{all} Banach $\Gamma$-modules.
Finally, we also prove a characterization of the (larger) class of \emph{boundedly acyclic} homomorphisms, that is, the class of group homomorphisms $\phi \colon \Gamma \to K$ for which the restriction maps in bounded cohomology $H^\bullet_b(K; V) \to H^\bullet_b(\Gamma; \phi^{-1}V)$ are isomorphisms for a suitable family of dual normed $\R[K]$-modules $V$ including the trivial $\R[K]$-module $\R$. We then extend the first and third results to topological spaces and obtain characterizations of \emph{amenable} maps and
\emph{boundedly acyclic} maps in terms of the vanishing of the bounded cohomology of their homotopy fibers with respect to appropriate choices 
of coefficients.
  \end{abstract}

\maketitle

\section{Introduction and Statement of Results}

Bounded cohomology $H^\bullet_b(-;-)$ was first introduced by Johnson~\cite{Johnson} and Trauber (unpublished) to address problems about Banach algebras. 
The theory of bounded cohomology grew into an independent and active research field after the pioneering work of Gromov~\cite{vbc} who extended the theory from groups to topological spaces, developed many of its fundamental properties, and explored its connections with geometry and group theory. 
In the setting of topological spaces, bounded cohomology is a functional-analytic variant of ordinary singular cohomology which involves only those singular cochains that have bounded norm. Building on Gromov's work, Ivanov~\cite{ivanov, Ivanov17} gave a detailed account of the foundations of bounded cohomology that also emphasized the use of methods from homological algebra; the case of twisted coefficients was treated also 
by Noskov~\cite{noskov}. The theory of bounded cohomology was developed further and extended to topological groups by Burger and Monod \cite{Burger_Monod_2002, monod}.

\medskip

Among the first striking results of bounded cohomology theory is the following characterization of amenable groups due to Johnson~\cite{Johnson}:
\emph{a discrete group $\Gamma$ is amenable if and only if $H^n_b(\Gamma; V) = 0$ for all dual normed} $\mathbb{R}[\Gamma]$-\emph{modules $V$ 
and} $n \geq 1$~(see also \cite[Corollary~3.11]{Frigerio:book}). We emphasize that this characterization holds specifically with respect to \emph{dual normed} $\R[\Gamma]$-modules. In other words, the bounded cohomology groups $H^n_b(\Gamma ; V)$ of an amenable group $\Gamma$ do not vanish in general for arbitrary normed $\R[\Gamma]$-modules (see, e.g.,~\cite{noskov}). On the other hand, the triviality of the bounded cohomology of $\Gamma$ with trivial coefficients in $\R$ does not suffice to characterize amenability; counterexamples were constructed by Matsumoto--Morita~\cite{Matsu-Mor} and L\"{o}h~\cite{Loeh:dim}.
In fact, the situation is even more delicate: Monod has recently shown that even the vanishing of the bounded cohomology for all \emph{separable} coefficients is not sufficient for a group to be amenable~\cite{monod:thompson}.

The central role of amenability in the theory of bounded cohomology is further elucidated by Gromov's \emph{Mapping Theorem}~\cite{vbc, Ivanov17, FM:Grom}. We state here a strong version of this theorem for maps $f \colon X \to Y$ and twisted coefficients on $X$~\cite[Corollary~2.8.5]{clara:book}; the complete proof can be obtained by combining, e.g.~\cite[Corollary~7.5.10]{monod}, \cite[Appendix~B]{Loehthesis} and~\cite[Corollary~6.4]{Ivanov17}. 
We will review the definition and basic properties of bounded cohomology in \textbf{Section~\ref{sec:bdd:cohomology}}.

\begin{thm*}[Mapping Theorem] \label{thm:mapping_thm}
Let $f \colon X \to Y$ be a map of based path-connected spaces such that $f_* \colon \pi_1(X) \to \pi_1(Y)$ is surjective with amenable kernel $H$.
Then, for all dual normed $\mathbb{R}[\pi_1(X)]$-modules $V,$ the induced inflation map
$$
H^\bullet_b(f ; \textup{I}_V) \colon H_b^\bullet(Y; V^H) \to H_b^\bullet(X; V)
$$
is an isometric isomorphism. Here $\textup{I}_V \colon V^H \to V$ denotes the inclusion of $H$-fixed points of $V$.
\end{thm*}

We call a map $f \colon X \to Y$ \emph{amenable} if it satisfies the conclusion of the \emph{Mapping Theorem} (see Definition \ref{def:amenable:maps}). As a consequence of Johnson's characterization of amenability and the \emph{Mapping Theorem}, we observe that 
a map $X \to *$ is amenable if and only if $X$ has amenable fundamental group. In this case, we also say that $X$ is \emph{amenable}. 

Our first goal in this paper is to prove that the assumptions of the \emph{Mapping Theorem} completely characterize amenable maps. This characterization may be regarded as a parametrized version of Johnson's characterization of amenable groups. To clarify this viewpoint, let us first observe  
the following homotopy-theoretic reformulation of the assumptions on $f_*$ in the \emph{Mapping Theorem} (cf. \cite[pp. 167--168]{monod}). Let $F$ denote the \emph{homotopy fiber of $f$}, that is, the fiber of a replacement of $f$ by a fibration up to (weak) homotopy equivalence. The (weak) homotopy type of $F$ is uniquely determined and there is a long exact sequence of homotopy groups (see, for example, \cite{tomDieck-AT}):
$$
\cdots \to \pi_2(Y) \to \pi_1(F) \to \pi_1(X) \xrightarrow{f_*} \pi_1(Y) \to \pi_0(F) \to \pi_0(X) \to \cdots \ .
$$
Since $X$ is path-connected, it follows that $F$ is path-connected if and only if $f_*$ is surjective. Moreover, $\pi_1(F)$ is a (central) extension 
of $\ker(f_*)$ 
$$
1 \to A \to \pi_1(F) \to \ker(f_*) \to 1$$
where the abelian group $A$ is given by the image of the homomorphism $\pi_2(Y) \to \pi_1(F)$. Therefore, $\pi_1(F)$ is amenable if and only if $\ker(f_*)$ is amenable~\cite[Proposition~3.4]{Frigerio:book}. The following theorem partly reformulates the \emph{Mapping Theorem} and also shows that the assumptions of the \emph{Mapping Theorem} completely characterize amenable maps.

\begin{introthm}\label{main:thm:intro:amenable} 
Let $f \colon X \to Y$ be a map of based path-connected spaces, let $f_* \colon \pi_1(X) \to \pi_1(Y)$ be the induced homomorphism between the fundamental groups and let $H$ denote its kernel. Let $F$ denote the homotopy fiber of $f$ and suppose that $F$ is path-connected (equivalently, $f_*$ is surjective). Then the following are equivalent:
\begin{enumerate}
\item[(1)] $f$ is an amenable map, that is, for all dual normed $\mathbb{R}[\pi_1(X)]$-modules $V$, the induced inflation map
$$
H^\bullet_b(f ; \textup{I}_V) \colon H_b^\bullet(Y; V^H) \to H_b^\bullet(X; V)
$$
is an isometric isomorphism.
\item[(2)] For all dual normed $\mathbb{R}[\pi_1(X)]$-modules $V$, the induced inflation map
$$
H^1_b(f ; \textup{I}_V) \colon H_b^1(Y; V^H) \to H_b^1(X; V)
$$
is an isomorphism.
\item[(3)] $F$ is amenable. 
\end{enumerate}
\end{introthm}

We note that the characterization in (3) essentially states that a map $f \colon X \to Y$ is amenable if and only if $f$ is \emph{fiberwise} amenable. 
Using the \emph{Mapping Theorem}, Theorem \ref{main:thm:intro:amenable} will be obtained from an analogous characterization of \emph{amenable} homomorphisms in the context of discrete groups (see Definition \ref{def:amenable:maps} and Theorem \ref{thm:main:groups:amenable}). More precisely, Theorem \ref{thm:main:groups:amenable} characterizes surjective group homomorphisms $\phi \colon \Gamma \to K$ with amenable kernel in terms of bounded cohomology and conditions analogous to Theorem \ref{main:thm:intro:amenable} -- the statement for discrete groups corresponds to the special case of Theorem \ref{main:thm:intro:amenable} for maps of the form $B\phi \colon B\Gamma \to BK$. The proof of this characterization for discrete groups is obtained from a ``relative" version of Johnson's characterization of amenability combined with the exact sequences in bounded cohomology that are induced by group extensions~\cite[Chapter 12]{monod}. We prove Theorem \ref{thm:main:groups:amenable} and then deduce Theorem \ref{main:thm:intro:amenable} in \textbf{Section~\ref{section:proof:amenable:maps}}. 

\medskip

It is well-known that the vanishing of the bounded cohomology groups $H^{n \geq 1}_b(\Gamma; V)$ for all Banach $\Gamma$-modules $V$ is a much stronger condition on $\Gamma$ which actually characterizes the class of \emph{finite} groups \cite[Section~3.5]{Frigerio:book}. In \textbf{Section \ref{section:proof:amenable:maps}}, we also prove the following generalization of this characterization to the context of group homomorphisms.

\begin{introthm}\label{main:thm:intro:finite} 
Let $\phi \colon \Gamma \to K$ be a surjective homomorphism of discrete groups and let $H$ denote the kernel of $\phi$. 
Then the following are equivalent:
\begin{enumerate}
\item[(1)] For all Banach $\Gamma$-modules $V$, the induced inflation map
$$
H^{\bullet}_b(\phi ; \textup{I}_V) \colon H_b^{\bullet}(K; V^H) \to H_b^{\bullet}(\Gamma; V)
$$
is an isometric isomorphism. 
\item[(2)] For all Banach $\Gamma$-modules $V$, the induced inflation map
$$
H^1_b(\phi ; \textup{I}_V) \colon H_b^1(K; V^H) \to H_b^1(\Gamma; V)
$$
is an isomorphism.
\item[(3)] $H$ is finite.
\end{enumerate}
\end{introthm}

We do not know a corresponding statement to Theorem B for topological spaces that would be analogous to Theorem \ref{main:thm:intro:amenable}, because the argument for passing from groups to topological spaces cannot be applied in the same way (see Remark~\ref{rem:finite:groups:spaces}).

\medskip

Once we have these characterizations, a natural question is to understand what happens if we only consider the weaker condition that the bounded cohomology of the homotopy fiber $F$ (resp. $H$) with coefficients in $\R$ vanishes in positive degrees. (For any group $\Gamma$, $\R$ is regarded as an
$\R[\Gamma]$-module with the trivial $\Gamma$-action.) Since there are many non-amenable groups with trivial bounded cohomology with $\R$-coefficients (e.g., the group of compactly supported homeomorphisms of $\mathbb{R}^n$~\cite{Matsu-Mor} and the mitotic groups \cite{Loeh:dim}), this class of maps (or homomorphisms) is strictly larger than the class of amenable maps (or homomorphisms). Additionally, in connection with the definition and properties
of acyclic maps in classical homotopy theory (see, for example, \cite{HH-acyclic, Raptis-acyclic}), it seems natural to identify the corresponding class of maps for bounded cohomology. We say that a map $f \colon X \to Y$ of based path-connected spaces is \emph{boundedly $n$-acyclic} if the induced restriction map
$$
H^i_b(f; V) \colon H_b^i(Y; V) \to H_b^i(X; f_*^{-1}V)
$$
is an isomorphism for $i \leq n$ and injective for $i = n + 1$, for every dual normed $\mathbb{R}[\pi_1(Y)]$-module $V$ of the form $\ell^{\infty}(S, \R)$
where $S$ is a $\pi_1(Y)$-set. A Banach $\Gamma$-module of this form will be called \emph{$\R$-generated}. The key property of these modules
is that every group with trivial bounded cohomology with $\R$-coefficients also has 
trivial bounded cohomology with respect to every $\R$-generated Banach $\Gamma$-module (Proposition~\ref{prop:trivial_mod}).
Moreover, we say that a topological space $X$ is boundedly $n$-acyclic if the map $X \to *$ is boundedly $n$-acyclic; equivalently, this means that $X$ is path-connected and $H^i_b(X; \R) = 0$ for $1 \leq i \leq n$ (see Proposition \ref{prop:trivial_mod}). In analogy with the characterizations of acyclic maps for singular cohomology, we obtain the following characterizations for the class of boundedly $n$-acyclic maps. 

\begin{introthm}\label{main:thm:intro} 
Let $f \colon X \to Y$ be a map of based path-connected spaces, let $F$ denote its homotopy fiber, and let $n \geq 0$ be an integer or $n = \infty$.  We denote by $f_* \colon \pi_1(X) \to \pi_1(Y)$ the induced homomorphism between the fundamental groups. Then the following are equivalent:
\begin{enumerate}
\item[(1)] $f$ is boundedly $n$-acyclic.
\item[(2)] The induced restriction map 
$$H^i_b(f; V) \colon H_b^i(Y; V) \to H_b^i(X; f_*^{-1}V)$$ 
is surjective for $0 \leq i \leq n$ and every $\R$-generated Banach $\pi_1(Y)$-module $V$.
\item[(3)] $F$ is path-connected and $H^i_b(X; f_*^{-1}V) = 0$ for $1 \leq i \leq n$ and every relatively injective $\R$-generated Banach $\pi_1(Y)$-module $V$.
\item[(4)]  $F$ is boundedly $n$-acyclic, that is, $H^0_b(F; \mathbb{R}) \cong \mathbb{R}$ and $H^i_b(F; \mathbb{R}) = 0$ for $1 \leq i \leq n$.
\end{enumerate}
\end{introthm}

The proof of Theorem~\ref{main:thm:intro} makes use of the properties of bounded cohomology purely from the viewpoint of homological algebra. 
More precisely, we introduce \emph{boundedly acyclic} resolutions, which are analogous to the classical notion of acyclic resolutions in ordinary cohomology.
The drawback of the purely algebraic viewpoint for bounded cohomology is that it neglects the metric structure of bounded cohomology. In particular, it seems likely that the statements in (3) and (4) of Theorem \ref{main:thm:intro} fail in general to characterize the maps $f \colon X \to Y$ for  which the induced restriction maps are \emph{isometric} isomorphisms in bounded cohomology. Still, it is reasonable to make full and unreserved use of this viewpoint and of the corresponding well-established methods of homological (and homotopical) algebra, especially, given that relatively few calculations of bounded cohomology groups are actually known. We hope that working with boundedly ($n$-)acyclic maps and the characterizations in Theorem~\ref{main:thm:intro} can lead to new and interesting examples of bounded cohomology equivalences outside situations which involve amenable groups. Note that restriction maps are a special case of inflation maps (Remark~\ref{rem:inflation:vs:pullback}); this also clarifies the connection between Theorem~\ref{main:thm:intro:amenable} and Theorem~\ref{main:thm:intro}. 

Using the \emph{Mapping Theorem}, the proof of Theorem~\ref{main:thm:intro} will again be deduced from an analogous statement which characterizes the homomorphisms $\phi\colon \Gamma \to K$ of discrete groups such that the restriction map 
$$H^{\bullet}_b(\phi; V) \colon H^\bullet_b(K; V) \to H^\bullet_b(\Gamma; \phi^{-1}V)$$
is an isomorphism for $i \leq n$ and injective for $i = n + 1$, for every $\R$-generated $K$-module $V$ (Theorem \ref{thm:main:groups}). 
Similarly to Theorem \ref{main:thm:intro} (4), these turn out to be exactly the surjective homomorphisms $\phi \colon \Gamma \to K$ whose kernel $H$ is \emph{boundedly $n$-acyclic}, that is, $H^i_b(H; \R) = 0$ for $1 \leq i \leq n$.  Our indexing convention in the definition of boundedly $n$-acyclic maps/homomorphisms is made to agree with the vanishing range of the bounded cohomology of the respective homotopy fiber/kernel. 

Some known examples of boundedly $n$-acyclic groups are the amenable groups (by the \emph{Mapping Theorem}), the group of compactly supported homeomorphisms of $\mathbb{R}^n$~\cite{Matsu-Mor} and the mitotic groups \cite{Loeh:dim} (for $n = \infty$), as well as lattices in higher rank Lie groups~\cite{Burger_Monod_99} and groups with dynamical properties~\cite{fflodha} (known for finite $n$). The study of bounded acyclicity has attracted renewed interest recently, and many new examples have been discovered. The results of Matsumoto--Morita \cite{Matsu-Mor} and L\"oh \cite{Loeh:dim} were generalized and unified in the proof that \emph{all} binate groups are boundedly acyclic~\cite{fflm2}. Moreover, finitely presented non-amenable boundedly acyclic groups have also been constructed~\cite{fflm1}. Furthermore, Monod \cite{monod:thompson} recently proved the boundedly acyclicity of Thompson's group $F$ and of \emph{lamplighter groups}, and Monod--Nariman \cite{monodnariman} showed examples of homeomorphism and diffeomorphism groups which are boundedly acyclic. It is worth mentioning that these new examples of boundedly acyclic groups, together with the study of bounded acyclicity in general, also led to the first computations of the full bounded cohomology rings of certain groups, including Thompson's group $T$~\cite{fflm2, monod:thompson} and 
the group of orientation-preserving homemorphisms of $S^1$~\cite{monodnariman}.

In \textbf{Section \ref{sec:proofs}},  we prove Theorem~\ref{thm:main:groups} and then deduce Theorem~\ref{main:thm:intro}. Moreover, we use these characterizations to discuss the closure properties of the class of boundedly $n$-acyclic groups and the stability properties of boundedly $n$-acyclic maps under various standard homotopy-theoretic operations. Finally, we discuss an application to the vanishing of the relative simplicial volume of an oriented compact connected $n$-manifold $(M, \partial M)$, assuming that the inclusion $\partial M \to M$ is boundedly $(n-1)$-acyclic. 


\medskip

\noindent \textbf{Acknowledgements.} We thank Clara L\"oh for her helpful and interesting comments on this work. We also thank Francesco Fournier-Facio for useful discussions about boundedly acyclic groups and for his comments on an earlier version of the paper. We are grateful to the anonymous referees for their careful reading and detailed comments. 

This work was partially supported by \emph{SFB 1085 -- Higher Invariants} (University of Regensburg) funded by the DFG. 


\section{Preliminaries}\label{sec:bdd:cohomology}

\subsection{Bounded cohomology}\label{sec:bc}
We briefly recall the definition and the basic properties of bounded cohomology of groups and spaces. We refer the reader to
Gromov~\cite{vbc}, Ivanov~\cite{ivanov, Ivanov17}, Frigerio~\cite{Frigerio:book}, L\"{o}h~\cite{clara:book} and Monod~\cite{monod}, for 
detailed expositions of the theory. 

\medskip

Let $\Gamma$ be a discrete group. An $\mathbb{R}[\Gamma]$-\emph{module} is a \emph{normed} $\R$-module $V \coloneqq (V, \sv{\cdot}_V)$ equipped with an action of $\Gamma$ by $\R$-linear isometries. Similarly, a Banach $\Gamma$-module is a \emph{complete} $\mathbb{R}[\Gamma]$-module. We say that $V$ is a \emph{trivial $\R[\Gamma]$-module} if the group $\Gamma$ acts trivially on $V$. 

\medskip
 
Let $X \coloneqq (X, x)$ be a based path-connected topological space and let $\Gamma \coloneqq  \pi_1(X, x)$ denote its fundamental group. \emph{We restrict throughout this paper to based topological spaces $X$ equipped with a universal covering} $p \colon 
\widetilde{X} \to X$ \emph{and a basepoint} $\widetilde{x} \in p^{-1}(x)$. The group $\Gamma$ acts continuously on $\widetilde{X}$ by deck transformations. We consider the real singular chain complex 
$C_\bullet(\widetilde{X}; \R)$ as a normed $\mathbb{R}[\Gamma]$-module endowed with the $\ell^1$-norm. We recall that the $\ell^1$-\emph{norm} of a singular reduced $\R$-chain $c = \sum_{i = 1}^k \alpha_i \sigma_i \in \, C_\bullet(\widetilde{X}; \R)$ is defined by $\sv{c}_1 = \sum_{i = 1}^k |\alpha_i|$. 
Here a chain is said to be in \emph{reduced form} if $\sigma_i \neq \sigma_j$ for all $i \neq j$.

Let $V$ be  a normed $\R[\Gamma]$-module. Then, we can consider the normed $\mathbb{R}[\Gamma]$-module of bounded operators $\mathcal{B}(C_\bullet(\widetilde{X}; \R), V)$ (with the $\ell^{\infty}$-norm); this is equipped with an action of $\Gamma$ given by  
$$
(g \cdot f)(c) = g \cdot f (g^{-1} c),
$$
where $g \in \, \Gamma$, $f \in \, \mathcal{B}(C_\bullet(\widetilde{X}; \R), V)$ and $c \in \, C_\bullet(\widetilde{X}; \R)$. We then denote by $C_b^\bullet(X; V)$ the (normed) space of $\Gamma$-invariants $\mathcal{B}(C_\bullet(\widetilde{X}; \R), V)^\Gamma$ of $\mathcal{B}(C_\bullet(\widetilde{X}; \R), V)$. Since the coboundary operators $\delta^\bullet$ preserve $\Gamma$-invariant bounded cochains, we obtain the \emph{bounded cochain complex} $(C^\bullet_b(X; V), \delta^\bullet)$ \emph{of $X$ with coefficients in $V.$}

\begin{defi}[Bounded cohomology of spaces]
Let $X$, $\Gamma$ and $V$ be as above. The \emph{bounded cohomology of} $X$ \emph{with coefficients in $V$}, denoted by $H^\bullet_b(X; V),$ is the cohomology of the bounded cochain complex 
$(C^\bullet_b(X; V), \delta^\bullet)$.
\end{defi}

\begin{rem}[Seminorm on bounded cohomology]
The bounded cohomology groups are canonically endowed with a seminorm that is induced by the $\ell^\infty$-norm. 
We will say that an isomorphism in bounded cohomology is \emph{isometric} if it preserves the $\ell^\infty$-seminorm.
\end{rem}

Using this definition of bounded cohomology for spaces, it is possible to define the bounded cohomology $H^\bullet_b(\Gamma; V)$ of a discrete group $\Gamma$ to be the corresponding bounded cohomology of a model for the classifying space $B\Gamma$. This definition is independent of the choice of a model because bounded cohomology is homotopy invariant~\cite{Ivanov17}. But it might also be useful to recall the standard definition of bounded cohomology of groups using the \emph{standard resolution} (see, for example, \cite{Frigerio:book}) as this will be needed later. Let $\Gamma$ be a discrete group and let  $V$ be a normed $\mathbb{R}[\Gamma]$-module. We define 
$$C^\bullet(\Gamma; V) \coloneqq \{ f \colon \Gamma^{\bullet+1} \to V\}$$
to be the group of $n$-cochains on $\Gamma$ with coefficients in $V$; these groups form the standard resolution $(C^{\bullet}(\Gamma; V), \delta^{\bullet})$ that is used in the definition of the group cohomology of $\Gamma$ with coefficients in $V.$  As before, we restrict to the \emph{bounded cochains} $C_b^\bullet(\Gamma; V) \subset C^\bullet(\Gamma; V),$
$$C_b^\bullet(\Gamma; V) \coloneqq \{f \in \, C^\bullet(\Gamma; V) \, | \, \sv{f}_\infty < \infty\},$$
where $\sv{f}_\infty \coloneqq \sup_{g_0, \cdots, g_\bullet} \sv{f(g_0, \cdots,g_\bullet)}_V$ is the $\ell^{\infty}$-norm.
Finally, we denote by $$C_b^\bullet(\Gamma; V)^\Gamma \subset C_b^\bullet(\Gamma; V)$$
the $\Gamma$-invariant bounded functions; here, $C_b^\bullet(\Gamma; V)$
is endowed with the corresponding diagonal $\Gamma$-action:
$$
(g \cdot f )(g_0, \cdots, g_\bullet) \coloneqq g \cdot f(g^{-1}g_0, \cdots, g^{-1}g_\bullet).
$$
Note that the coboundary operators $\delta^{\bullet}$ of $C^{\bullet}(\Gamma; V)$ preserve $\Gamma$-invariant bounded cochains, so we obtain the \emph{bounded cochain complex} $(C^\bullet_b(\Gamma; V)^{\Gamma}, \delta^\bullet)$ 
\emph{of $\Gamma$ with coefficients in $V.$}

\begin{defi}[Bounded cohomology of groups]
Let $\Gamma$ and $V$ be as above. We define the \emph{bounded cohomology of $\Gamma$ with coefficients in $V$}, denoted by
$H_b^\bullet(\Gamma; V)$, to be the cohomology of the cochain complex of bounded $\Gamma$-invariant cochains
$(C_b^\bullet(\Gamma; V)^\Gamma, \delta^{\bullet})$.
\end{defi}

\begin{rem}[Dual normed modules]\label{rem:dual:modules}
Even though bounded cohomology with twisted coefficients is defined for \emph{arbitrary} normed $\R[\Gamma]$-modules $V$, 
it is common in the literature to restrict to Banach $\Gamma$-modules or further to \emph{dual} normed $\R[\Gamma]$-modules. We recall that a normed $\mathbb{R}[\Gamma]$-module $V$ is a \emph{dual} normed $\mathbb{R}[\Gamma]$-module if there exists a normed $\mathbb{R}[\Gamma]$-module $W$ such that $V$ is isomorphic to the topological dual of $W$ (as $\mathbb{R}[\Gamma]$-modules). Note that a dual normed $\R[\Gamma]$-module is always a Banach $\Gamma$-module. 

The restriction to dual normed modules is standard in the literature because it provides a convenient setting for many of the methods and techniques used in the theory in connection with actions by amenable groups \cite{Frigerio:book, Ivanov17, monod}. However, it would be useful to understand more precisely the  importance of the restriction to dual normed $\R[\Gamma]$-modules. For example, we do not know if the bounded cohomology of an amenable group $\Gamma$ vanishes for general normed \emph{trivial} $\R[\Gamma]$-modules. 
\end{rem}


\subsection{Restriction maps} Bounded cohomology  $H^{\bullet}_b(X; V)$ (and $H^{\bullet}_b(\Gamma; V)$) has several functoriality properties with respect to $X$ and $V.$ For a fixed $X$ (or $\Gamma$), bounded cohomology  $H^{\bullet}_b(X; -)$ (or $H^{\bullet}_b(\Gamma; -)$) is clearly (covariantly) functorial with respect to bounded linear $\Gamma$-maps, and defines a functor with values in the category of graded seminormed $\R$-modules and bounded linear maps. 

In this subsection, we recall some details about the functoriality of bounded cohomology in $X$ (and $\Gamma$) given by the \emph{restriction homomorphisms}.

\begin{defi}
Let $\Gamma$ and $K$ be discrete groups and let $\phi \colon \Gamma \to K$ be a homomorphism. For every 
normed $\mathbb{R}[K]$-module $V$, we define the \emph{pullback} (or restricted) normed $\mathbb{R}[\Gamma]$-module $\phi^{-1}V$
to be the normed module $V$ equipped with the 
$\Gamma$-action:
$$
g \cdot v = \phi(g) \cdot v 
$$
for every $g \in \, \Gamma$ and $v \in \, V$. Then precomposition with the homomorphism $\phi \colon \Gamma \to K$ induces the \emph{restriction map}:
$$
\phi^* \coloneqq H^\bullet_b(\phi; V) \colon H^\bullet_b(K; V) \to H^\bullet_b(\Gamma; \phi^{-1}V).
$$
\end{defi}

\begin{nota}\label{rem:restriction:map}
In the special case of an inclusion $\iota \colon H \to \Gamma$, following Monod's notation~\cite[Section~8.6]{monod}, we will write 
$$
\Res^\bullet \colon C^\bullet_b(\Gamma; V)^\Gamma \to C^\bullet_b(H; V)^H
$$
for the restriction map of cochain complexes, given by 
$$\Res^\bullet(\varphi)(h_0, \cdots, h_\bullet) \coloneqq \varphi(\iota(h_0), \cdots, \iota(h_\bullet));$$ 
note that $V$ denotes here both an $\R[\Gamma]$-module and the restricted $\R[H]$-module $\iota^{-1}V$. This map induces the corresponding restriction  map and will be denoted by
$$
\res^\bullet \colon H^\bullet_b(\Gamma; V) \to H^\bullet_b(H; V).
$$
Moreover, it is worth noticing that given a group extension
$$
1 \to H \xrightarrow{\iota} \Gamma \to K \to 1
$$
the restriction map $\res^\bullet \colon H^\bullet_b(\Gamma; V) \to H^\bullet_b(H; V)$ factors through the $K$-invariants $H^\bullet_b(H; V)^K$, where
$K$ acts by conjugation~\cite[Corollary~8.7.4]{monod}. 
\end{nota}

\begin{defi}
Let $f \colon X \to Y$ be a (based) map of based path-connected spaces, let $f_* \colon \pi_1(X) \to \pi_1(Y)$ denote the induced homomorphism, and let $\widetilde{f} \colon (\widetilde{X}, \widetilde{x}) \to (\widetilde{Y}, \widetilde{y})$ be the unique based lift of the map $f$. For every normed $\R[\pi_1(Y)]$-module $V$, precomposition with the $\pi_1$-equivariant map $\widetilde{f}$ induces the \emph{restriction map}:
$$
f^* \coloneqq H^\bullet_b(f; V) \colon H^\bullet_b(Y; V) \to H^\bullet_b(X; f_*^{-1}V).
$$
\end{defi}

The restriction homomorphisms in bounded cohomology are analogous to the corresponding homomorphisms in singular cohomology with local coefficients. In analogy with the class of acyclic maps in homotopy theory \cite{HH-acyclic, Raptis-acyclic}, there is a corresponding class of maps (or homomorphisms) which induce isomorphisms in bounded cohomology with respect to the restriction homomorphisms. However, certain subtle features of the category of normed or Banach $\Gamma$-modules require us to specialize the definition to suitable classes of coefficients. Recall that a set $S$ equipped with an action by a (discrete) group $\Gamma$ is called $\Gamma$-\emph{set}. Given a $\Gamma$-set $S$ and a Banach $\Gamma$-module $W$, the space of $W$-valued bounded functions $\ell^\infty(S, W)$ is a Banach $\Gamma$-module endowed with the action
$$
(g \cdot f)(s) = gf(g^{-1}s)
$$
for all $g \in \, \Gamma$, $s \in \, S$ and $f \in \, \ell^\infty(S, W)$.

\begin{defi}
Let $\Gamma$ be a discrete group. 
\begin{itemize}
\item[(a)] A Banach $\Gamma$-module $V$ is called \emph{$\R$-generated} if it is of the form $\ell^{\infty}(S, \R)$ for a $\Gamma$-set $S$. 
\item[(b)] Let $W$ be a Banach $\Gamma$-module. A \emph{$W$-generated Banach $\Gamma$-module} is a Banach $\Gamma$-module $V$ 
of the form $\ell^{\infty}(S, W)$ where $S$ is a $\Gamma$-set. 
\end{itemize}
\end{defi}

\begin{rem}\label{rem:generated:module:examples}
Every Banach $\Gamma$-module $W$ is also a $W$-generated Banach $\Gamma$-module. Moreover, given a group homomorphism 
$\phi \colon \Gamma \to K$ and a $W$-generated Banach $K$-module $V$, then the restricted $\R[\Gamma]$-module $\phi^{-1}V$ is a $\phi^{-1}W$-generated Banach $\Gamma$-module. Every $\R$-generated Banach $\Gamma$-module is also a \emph{dual} normed $\R[\Gamma]$-module. 
\end{rem}

\begin{defi}[Boundedly $n$-acyclic maps/homomorphisms]\label{def:intro:bdd:cohom:equiv} Let $n \geq 0$ be an integer or $n=\infty$. 
\begin{itemize}
\item[(a)] A map $f \colon X \to Y$ of based path-connected spaces is \emph{boundedly $n$-acyclic} if the restriction map
$$
H^i_b(f; V) \colon H_b^i(Y; V) \to H_b^i(X; f_*^{-1}V)
$$
is an isomorphism for $i \leq n$ and injective for $i = n + 1$ for every $\R$-generated Banach $\pi_1(Y)$-module $V$. 
\item[(b)] A homomorphism $\phi \colon \Gamma \to K$ of discrete groups is \emph{boundedly $n$-acyclic} if the restriction map
$$
H^i_b(\phi; V) \colon H_b^i(K; V) \to H_b^i(\Gamma; \phi^{-1}V)
$$
is an isomorphism for $i \leq n$ and injective for $i = n+1$ for every $\R$-generated Banach $K$-module $V$. 
\item[(c)] We say that a map $f \colon X \to Y$ of based path-connected spaces (resp. a group homomorphism $\phi \colon \Gamma \to K$) is \emph{boundedly acyclic} if it is boundedly $\infty$-acyclic. 
\end{itemize}
\end{defi}

Note that the class of boundedly $n$-acyclic maps/homomorphisms is closed under composition. The following proposition shows the close connection between boundedly $n$-acyclic maps and boundedly $n$-acyclic homomorphisms. 

\begin{prop} \label{prop:compare_bdd_acyclic} Let $n \geq 0$ be an integer or $n = \infty$. 
A map $f \colon X \to Y$ of based path-connected spaces is boundedly $n$-acyclic if and only if the induced homomorphism $f_* \colon \pi_1(X) \to \pi_1(Y)$ is boundedly $n$-acyclic. 
\end{prop}
\begin{proof}
Consider the following diagram of spaces 
\begin{equation*}\label{eq:diagram:main:thm}
\xymatrix{
X \ar[rr]^-{c_{X}} \ar[d]^-{f} && B\pi_1(X) \ar[d]^{Bf_*} \\
Y \ar[rr]_{c_Y} && B\pi_1(Y)
}
\end{equation*}
where the horizontal maps $c_{X}$ and $c_Y$ are the canonical maps to (some functorial model of) the classifying spaces of the respective fundamental groups. By the \emph{Mapping Theorem} (with coefficients) \cite{vbc, Ivanov17, FM:Grom, clara:book}, the restriction homomorphisms associated to the horizontal maps $c_{X}$ and $c_Y$ induce (isometric) isomorphisms in bounded cohomology 
$$H^\bullet_b(B\pi_1(-); V) \xrightarrow{\cong} H^\bullet_b(-; V)$$ for all $\R$-generated (or dual normed) Banach $\pi_1$-modules $V$. Note that since the classifying maps are $\pi_1$-isomorphisms, $V$ is isomorphic as $\R[\pi_1]$-module to its pullback.
So, $f$ is boundedly $n$-acyclic if and only if $Bf_*$ is boundedly $n$-acyclic, and therefore the required result follows. 
\end{proof}

\begin{rem}\label{rem:2:out:3:boundedly:acyclic}
The class of boundedly $n$-acyclic maps/homomorphisms is closed under composition, but it does not satisfy the 2-out-of-3 property: there are composable homomorphisms $\Gamma \xrightarrow{\phi} K \xrightarrow{\phi'} L$ such that $\phi'$ and $\phi'\phi$ are boundedly $n$-acyclic but $\phi$ is not (e.g., for $\Gamma$ and $L$ 
the trivial group and $K$ a non-trivial amenable group -- note the surjectivity assumption in Theorem~\ref{thm:main:groups}). On the other hand, $\phi'$ is boundedly $n$-acyclic if both $\phi$ and $\phi' \phi$ are boundedly $n$-acyclic.  
\end{rem}

\begin{rem} \label{W-bdd-acyclic}
There are analogous definitions in the case of $W$-generated Banach modules. Let $n \geq 0$ be an integer or $n=\infty$, let $f \colon X \to Y$ be a map of based path-connected spaces and let $W$ be a Banach $\pi_1(Y)$-module. The map $f$ is \emph{$\langle W \rangle$-boundedly $n$-acyclic} if the restriction map
$$
H^i_b(f; V) \colon H_b^i(Y; V) \to H_b^i(X; f_*^{-1}V)
$$
is an isomorphism for $i \leq n$ and injective for $i = n + 1$ for every $W$-generated Banach $\pi_1(Y)$-module $V$. Definition \ref{def:intro:bdd:cohom:equiv} 
corresponds to the case $W = \R$. We do not know if a boundedly acyclic map is also $\langle W \rangle$-boundedly acyclic for general $W$. 
\end{rem}


\subsection{Induction modules} 
We recall the analogue of the Eckmann--Shapiro lemma for bounded cohomology~(see~\cite[Section~10.1]{monod} and Remark~\ref{rem:coeff:module}). 
This identifies the bounded cohomology of a subgroup $H \leq \Gamma$ with the bounded cohomology of $\Gamma$ with coefficients in the induction module. 

\begin{defi}[Induction module]\label{def:induction:module}
Let $\Gamma$ be a discrete group and let $H \leq \Gamma$ be a subgroup. Let $V$ be a (dual) normed $\mathbb{R}[H]$-module. 
The \emph{induction module} $\textbf{I}_H^{\Gamma} V$ is the (dual) normed $\R[\Gamma]$-module $\ell^\infty(\Gamma, V)^H$ of $H$-invariant bounded 
functions $f \colon \Gamma \to V$; the $\Gamma$-action on $\textbf{I}_H^{\Gamma} V$ is given by right translation, i.e.,
$(g \cdot f)(x) = f(x g)$ for $g, x \in \, \Gamma$.
\end{defi}

\begin{rem}\label{rem:induction:module:action:Monod}
Suppose that $V$ is a dual normed $\R[\Gamma]$-module, instead of just a dual normed $\R[H]$-module. 
Then, the dual normed $\R[\Gamma]$-module $\textbf{I}_H^{\Gamma} V$ is isomorphic
to $\ell^\infty(\Gamma/H, V)$ endowed with the following $\Gamma$-action:
$$
(g \cdot f)(k) = g f(g^{-1} k) 
$$
for all $k \in \, \Gamma/H$, $g \in \, \Gamma$ and $f \in \, \ell^\infty(\Gamma/H, V)$~\cite[Remark~10.1.2(v)]{monod}.
In particular, $\textbf{I}_H^{\Gamma} V$ is a $V$-generated Banach $\Gamma$-module.
\end{rem}

\begin{prop}[Eckmann--Shapiro lemma~{\cite[Proposition~10.1.3]{monod}}] \label{prop:shapiro1}
Let $\Gamma$ be a discrete group and let $H \leq \Gamma$ be a subgroup. For every dual normed $\R[H]$-module $V$, there is a canonical isometric isomorphism: $$\textbf{i}^\bullet \colon H^\bullet_b(H; V) \xrightarrow{\cong} H^\bullet_b(\Gamma; \textbf{I}_H^{\Gamma} V).$$
\end{prop}

The construction of the induction module has several naturality properties (see \cite[Proposition~10.1.5]{monod}). Given a diagram of homomorphisms between discrete groups
\begin{equation} \tag{$\star$} \label{squares}
\xymatrix{
H \ar@{}[r]|-*[@]{\subset} \ar[d]_{\phi |_H} & \Gamma \ar[d]^{\phi} \\
H' \ar@{}[r]|-*[@]{\subset}  & K 
}
\end{equation}
where the horizontal maps are inclusions of subgroups, and given a dual normed $\mathbb{R}[H']$-module $V$, then there is a natural transformation of dual normed $\R[\Gamma]$-modules
$$
\Psi \colon \phi^{-1} \mathbf{I}_{H'}^K V \to \mathbf{I}_{H}^\Gamma {(\phi |_H)}^{-1} V, \ \ f \to f \circ \phi,
$$
where $f \in \, \ell^\infty(K, V)^{H'}$. We record the following special case of $\Psi$ for later use.

\begin{prop}\label{prop:coeff:isomorphic}
Let $\phi \colon \Gamma \to K$ be a surjective homomorphism of discrete groups and let $H$ denote its kernel. 
Assume that $V$ is a dual normed $\mathbb{R}$-module. Then the map
 $$
 \Psi \colon \phi^{-1} \mathbf{I}_{\{1\}}^K V \to \mathbf{I}_{H}^\Gamma {(\phi |_H)}^{-1} V 
 $$
is an isometric isomorphism of $V$-generated Banach $\Gamma$-modules.
\end{prop}
\begin{proof} Unwinding the definitions, the comparison map $\Psi$ is identified in this case with the canonical isometric isomorphism of dual normed 
$\R[\Gamma]$-modules $\ell^{\infty}(K, V) \xrightarrow{\cong} \ell^{\infty}(\Gamma, V)^H, \  f \mapsto f \circ \phi$, where $V$ is endowed with the trivial action.
\end{proof}

The isomorphism of Proposition \ref{prop:shapiro1} is natural with respect to diagrams~\eqref{squares} and restriction homomorphisms:

\begin{prop}[{see~\cite[Proposition~10.1.5 (iv)]{monod}}]\label{prop:shapiro2}
Consider a diagram of group homomorphisms 
$$
\xymatrix{
H \ar@{}[r]|-*[@]{\subset} \ar[d]_{\phi |_H} & \Gamma \ar[d]^{\phi} \\
H' \ar@{}[r]|-*[@]{\subset}  & K 
}
$$
where the horizontal maps are inclusions of subgroups and let $V$ be a dual normed $\mathbb{R}[H']$-module. Then the following diagram commutes:
$$
\xymatrix@C=1em{
H^\bullet_b(H'; V) \ar[rr]^{\cong}_{\textbf{i}^\bullet} \ar[dd]_-{(\phi |_H)^*} && H^\bullet_b(K; \mathbf{I}_{H'}^{K} V) \ar[rrd]^-{\phi^*} \\
&&&& H^\bullet_b(\Gamma; \phi^{-1} \mathbf{I}_{H'}^{K} V) \ar[dll]^-{H^\bullet_b(\textup{id}_\Gamma; \Psi)} \\
H^\bullet_b(H; (\phi |_H)^{-1} V) \ar[rr]^(.45){\cong}_(.45){\textbf{i}^\bullet} && H^\bullet_b(\Gamma; \mathbf{I}_H^\Gamma {(\phi |_H)}^{-1} V).
}
$$
\end{prop}
\begin{rem}\label{rem:all:coeff:shapiro:are:generated}
Notice that if $V$ is a trivial dual normed $\R[H']$-module, then all the coefficients appearing in the previous diagram are $V$-generated Banach modules (see Remark~\ref{rem:generated:module:examples} and Proposition~\ref{prop:coeff:isomorphic}).
\end{rem}
\begin{rem}\label{rem:coeff:module}
All the previous statements can be found in the work of Monod~\cite{monod} in the case of a \emph{second countable} locally compact group $\Gamma$ and a coefficient $\Gamma$-module $V$, i.e. a dual Banach $\Gamma$-module that is the topological dual of a \emph{separable} Banach $\Gamma$-module. However, in the case of discrete groups both the condition that $\Gamma$ is countable as well as the separability of the predual of $V$ can be omitted. Indeed, in the discrete case, the theory of amenable actions and of strong resolutions via relatively injective modules generally becomes simpler and can be applied without additional assumptions. See~\cite[Section~4.9]{Frigerio:book} and~\cite[Section~2.5]{monod} for more details.
\end{rem}

\subsection{Inflation maps} Bounded cohomology $H^*_b(X; V)$ (or $H^*_b(\Gamma; V)$) has additional functoriality properties 
in $X$ (or $\Gamma$) given by the \emph{inflation homomorphisms}. These provide a refinement of the restriction homomorphisms of the previous subsection. In this subsection, we recall briefly some details about the functoriality of bounded cohomology given by these maps.

\begin{defi}
Let $\Gamma$ be a discrete group and let $H \unlhd \Gamma$ be a normal subgroup. 
Given a normed $\mathbb{R}[\Gamma]$-module $V$, the normed module $V^H$ of $H$-fixed points inherits the structure of an $\R[\Gamma]$-module from $V$; moreover, this descends to an $\mathbb{R}[\Gamma \slash H]$-module. The inclusion $\textup{I}_V \colon V^H \to V$ is a map of $\R[\Gamma]$-modules. The restriction along the quotient homomorphism $\phi \colon \Gamma \to K \coloneqq \Gamma \slash H$ together with the map $\textup{I}_V \colon V^H \to V$ of $\R[\Gamma]$-modules induce the \emph{inflation map}: 
$$
\phi^{\#} \coloneqq H^\bullet_b(\phi; \textup{I}_V) \colon H^\bullet_b(K; V^H) \to H^\bullet_b(\Gamma; V) .
$$
\end{defi}

\begin{rem}\label{rem:inflation:vs:pullback}
Given a surjective homomorphism $\phi \colon \Gamma \to \Gamma \slash H$ and a normed $\mathbb{R}[\Gamma \slash H]$-module $V$, the inflation map $H^\bullet_b(\phi; \textup{I}_{\phi^{-1}V})$ agrees with the restriction map $H^\bullet_b(\phi; V)$. 
\end{rem}

Inflation and restriction are related in low degrees via a useful exact sequence shown in \cite[Section~12.4]{monod}.  We will only need the following weak form of this exact sequence. 
 
\begin{prop}\label{prop:sequence:restriction:inflation}
Let $\phi \colon \Gamma \to K$ be a surjective homomorphism of discrete groups with kernel $H$ and let $V$ be a Banach $\Gamma$-module. Then, the following sequence is exact:
$$
0 \to H^1_b(K; V^H) \xrightarrow{H^1_b(\phi; \textup{I}_V)} H^1_b(\Gamma; V) \xrightarrow{\res^1} H^1_b(H; V)
$$
\end{prop}
\begin{proof}
This is the first part of a longer exact sequence which extends to degrees $\leq 2,3$ \cite[Theorem~12.4.2]{monod}. Note that for this extended exact sequence, $H^1_b(H; V)$ must be replaced by the $K$-invariants of $H^1_b(H; V)$ with respect to the action that is induced by conjugation.
\end{proof}

\begin{cor}\label{rem:1:amenable:1:restriction}
Let $\phi \colon \Gamma \to K$ be a surjective homomorphism of discrete groups with kernel $H$ and let $V$ be a Banach $\Gamma$-module. If the inflation map
$
H^1_b(\phi; \textup{I}_V) \colon H^1_b(K; V^H) \to H^1_b(\Gamma; V)
$
is an isomorphism, then $\res^1 \colon H^1_b(\Gamma; V) \to H^1_b(H; V)$ is the zero map.
\end{cor}

\begin{defi}\label{def:inflation:spaces}
Let $f \colon X \to Y$ be a (based) map of based path-connected spaces, let $f_* \colon \pi_1(X) \to \pi_1(Y)$ denote the induced homomorphism, and let $\widetilde{f} \colon (\widetilde{X}, \widetilde{x}) \to (\widetilde{Y}, \widetilde{y})$ be the unique based lift of the map $f$. Suppose that $f_*$ is surjective and let $H$ denote its kernel. For every normed $\R[\pi_1(X)]$-module $V$, precomposition with $\widetilde{f}$ together with the map $\textup{I}_V \colon V^H \to V$ of $\R[\pi_1(X)]$-modules  induce the \emph{inflation map}:
$$
f^{\#} \coloneqq H^\bullet_b(f; \textup{I}_V) \colon H^\bullet_b(Y; V^H) \to H^\bullet_b(X; V).
$$
\end{defi}

Similarly to the case of boundedly acyclic maps (or homomorphisms) (Definition~\ref{def:intro:bdd:cohom:equiv}), we consider the corresponding class of maps (or homomorphisms) which induce isomorphisms in bounded cohomology with respect to the inflation homomorphisms. This is motivated by the \emph{Mapping Theorem} (with coefficients) which provides examples of such maps.

\begin{defi}[Amenable maps/homomorphisms]\label{def:amenable:maps} \
\begin{itemize}
\item[(a)] Let $f \colon X \to Y$ be a map as in Definition~\ref{def:inflation:spaces}. The map $f$ is \emph{amenable} if the inflation map
$$
H^\bullet_b(f; \textup{I}_V) \colon H^\bullet_b(Y; V^H) \to H^\bullet_b(X; V)
$$
is an isometric isomorphism for all dual normed $\mathbb{R}[\pi_1(X)]$-modules $V$. We say that a path-connected space $X$ is \emph{amenable} if 
the map $X \to *$ is amenable. 
\item[(b)] A surjective homomorphism $\phi \colon \Gamma \to K$ of discrete groups with kernel $H$ is \emph{amenable} if the inflation map
$$
H^\bullet_b(\phi; \textup{I}_V) \colon H^\bullet_b(K; V^H) \to H^\bullet_b(\Gamma; V)
$$
is an isometric isomorphism for all dual normed $\mathbb{R}[\Gamma]$-modules $V$. 
\end{itemize}
\end{defi}

\begin{rem} \label{rem:amenable_space}
By the characterization of amenability in terms of bounded cohomology, a group $\Gamma$ is amenable if and only if the trivial homomorphism $\Gamma \to 1$ is amenable. Moreover, using the \emph{Mapping Theorem}, a path-connected space $X$ is amenable if and only if $X$ has amenable fundamental group.
\end{rem}

\begin{prop}  \label{prop:compare_amenable}
Let $f \colon X \to Y$ be a map as in Definition~\ref{def:inflation:spaces}. Then $f \colon X \to Y$ is amenable if and only if $f_* \colon \pi_1(X) \to \pi_1(Y)$ is amenable. 
\end{prop}
\begin{proof}
The proof is the same as the proof of Proposition~\ref{prop:compare_bdd_acyclic}.
\end{proof}

\begin{rem}
The class of amenable maps/homomorphisms satisfies the 2-out-of-3 property: given composable maps $X \xrightarrow{f} Y \xrightarrow{g} Z$ between spaces (as in Def. \ref{def:inflation:spaces}), then the three maps $f, g,$ and $gf$ are amenable if any two of these are amenable. The proof uses Remark~\ref{rem:inflation:vs:pullback} for the less trivial case ($f, gf \text{ amenable} \Rightarrow g \text{ amenable}$).
\end{rem}

\subsection{Acyclic resolutions} If we ignore the metric structure, bounded cohomology theory of groups can be described as a 
(universal) $\delta$-functor in the classical sense of homological algebra. A foundational approach to bounded cohomology based 
on this observation is developed by B\"uhler \cite{Buehler-article, Bualg}. In this subsection, we recall the main algebraic properties 
of bounded cohomology from the viewpoint of (relative) homological algebra. 

\medskip

Let $\mathbf{Mod}_{\R}^{\Gamma}$ denote the category of normed $\R[\Gamma]$-modules and bounded linear $\Gamma$-maps. 
In addition, let $\Mod$ denote the category of complete seminormed $\R$-modules and bounded linear maps. We begin by recalling the
definition of \emph{relatively injective} normed $\R[\Gamma]$-modules.

\begin{defi}
Let $A$ and $B$ be two normed $\R[\Gamma]$-modules. A morphism $i \colon A \to B$ in $\mathbf{Mod}_{\R}^{\Gamma}$
is \emph{strongly injective} if there exists an $\R$-linear map $\sigma \colon B \to A$ such that $\sv{\sigma} \leq 1$
and $\sigma \circ i = \textup{Id}_A$. 

A normed $\R[\Gamma]$-module $V$ is \emph{relatively injective} if for every strongly injective morphism $i \colon A \to B$ and bounded linear $\Gamma$-map $\alpha \colon A \to V$, there exists a morphism 
$\beta \colon B \to V$ in $\mathbf{Mod}_{\R}^{\Gamma}$ such that $\beta \circ i = \alpha$ and $\sv{\beta} \leq \sv{\alpha}$.
\end{defi}

\begin{thm}\label{thm:basic:properties:bdd:coho}
Let $\Gamma$ be a discrete group. Then the familiy of functors $H^i_b(\Gamma; -) \colon \ModG \to \Mod,$ $i \in \N,$ satisfies the following properties:
\begin{enumerate}
\item (Normalization) $H^0_b(\Gamma; V) \cong V^\Gamma$ for all normed $\R[\Gamma]$-modules $V$.
\item (Vanishing) $H_b^i(\Gamma; V) = 0$ for every $i \in \, \N_{> 0}$ and every relatively injective normed $\R[\Gamma]$-module $V$.
\item  (Long exact sequence) Let $0 \to A \to B \to C \to 0$ be a short exact sequence of Banach $\Gamma$-modules. Then, there exists
a natural family of bounded maps $(\tau^\bullet)$ such that the following sequence
$$
\cdots \to H_b^\bullet(\Gamma; A) \to H_b^\bullet(\Gamma; B) \to  H_b^\bullet(\Gamma; C) \xrightarrow{\tau^\bullet}  H_b^{\bullet+1}(\Gamma; A) \to \cdots
$$
is exact.
\end{enumerate}
\end{thm}
\begin{proof} We include a few comments for the convenience of the reader and refer to B\"uhler~\cite{Buehler-article, Bualg} and Monod~\cite{monod} for 
detailed proofs. First, we recall that the bounded cohomology groups are always complete seminormed spaces. Then, (1) follows easily from the fact that $C^0_b(\Gamma; V)^\Gamma \cong V^\Gamma$. For (2), let $V$ be a relatively injective normed $\R[\Gamma]$-module and consider the resolution $0 \to V \xrightarrow{\textup{Id}_V} V \to 0$. Using the general properties of (strong) resolutions by relatively injective $\R[\Gamma]$-modules, it follows that this resolution can be used to compute the bounded cohomology with coefficients in $V$ (see, for example,~\cite[Corollary~4.15]{Frigerio:book}). As a consequence, $H_b^i(\Gamma; V) = 0$, for every $i \in \, \N_{> 0}$. (3): Since $\Gamma$ is a discrete group, every Banach $\Gamma$-module is continuous in the sense of \cite[Lemma~1.1.1]{monod}. Then the result is 
a special case of~\cite[Proposition~8.2.1]{monod}. See also~\cite[Section~4]{Buehler-article}.
\end{proof}

Using these properties, it follows by standard methods of homological algebra that bounded cohomology can also be computed using resolutions by 
\emph{boundedly acyclic modules} instead of the standard resolution (Section~\ref{sec:bc}) or other strong resolutions by relatively injective modules. We emphasize that allowing this flexibility has the drawback that it neglects the metric structure given by the seminorm on the bounded cohomology groups.   

\begin{defi}\label{defi:acyclic:mod:and:res}
 Let $\Gamma$ be a discrete group and let $n \geq 1$ be an integer or $n=\infty$. A normed $\R[\Gamma]$-module $V \in \ModG$ is \emph{boundedly $n$-acyclic} if $H^i_b(\Gamma; V) = 0$ for $1 \leq i \leq n$. We say that $V$ is \emph{boundedly acyclic} if $V$ is boundedly $\infty$-acyclic. 
 \end{defi}
 
 
\begin{prop} \label{prop:acyclic_res}
Let $\Gamma$ be a discrete group and let $n \geq 1$ be an integer or $n=\infty$. Let $V$ be a normed $\R[\Gamma]$-module and suppose that 
$$0 \to V = I(-1) \to I(0) \to I(1) \to \cdots $$ 
is an exact sequence in $\ModG$ such that each $I(j) \in \ModG$ is a Banach $\Gamma$-module and boundedly $(n-j)$-acyclic for every 
$0 \leq j \leq n -1$ . Consider the cochain complex
$$I^{\bullet, \Gamma} = (0 \to I(0)^{\Gamma} \to I(1)^{\Gamma} \to \cdots ).$$
Then there are preferred isomorphisms
$$ H^i(I^{\bullet, \Gamma}) \xrightarrow{\cong} H^i_b(\Gamma; V) \ \text{ for } 0 \leq i \leq n,$$
and a preferred injective map 
$$ H^{n+1}(I^{\bullet, \Gamma}) \hookrightarrow H^{n+1}_b(\Gamma; V).$$
 \end{prop}
\begin{proof} The idea of the proof is standard in homological algebra; we include the details for completeness. The assertion holds for $i = 0$. 
For every  $j \geq 0,$ let $Z^{j-1}I^\bullet$ denote the image of $\delta^{j-1} \colon I(j-1) \to I(j)$.
Since $I(j)$ is a Banach $\Gamma$-module and the resolution $0 \to V \to I^\bullet$ is exact, we also know that 
 the spaces $Z^{j-1}I^\bullet$ are Banach $\Gamma$-modules. This implies that the short exact sequences for $0 \leq j \leq n-1$:
$$0 \to Z^{j-1}I^\bullet \to I(j) \to Z^{j}I^{\bullet} \to 0$$
induce long exact sequences in bounded cohomology (Theorem~\ref{thm:basic:properties:bdd:coho}(3)):
\begin{align*}
\cdots \to H^i_b(\Gamma;  Z^{j}I^{\bullet}) \to H^{i+1}_b(\Gamma;  Z^{j-1}I^{\bullet}) &\to H^{i+1}_b(\Gamma; I(j)) \\
&\to H^{i+1}_b(\Gamma;  Z^{j}I^{\bullet}) \to \cdots 
\end{align*}
Moreover, using the fact that each $I(j)$ is boundedly $(n-j)$-acyclic, i.e. $H^i_b(\Gamma; I(j)) = 0$ for $1 \leq i \leq n-j$, we conclude that
$$
H^{i+1}_b(\Gamma; Z^{j-1}I^{\bullet}) \cong H^{i}_b(\Gamma; Z^{j}I^{\bullet})
$$
for $2 \leq i + 1 \leq n-j$. This shows that for every $1 \leq i \leq n$, we have isomorphisms
\begin{align*}
H^i_b(\Gamma; V) = H^i_b(\Gamma; Z^{-1}I^\bullet) \cong H^{i-1}_b(\Gamma; Z^0 I^{\bullet}) \cong \cdots \cong H^1_b(\Gamma; Z^{i-2}I^{\bullet}).
\end{align*}
On the other hand, the long exact sequence in low degrees has the form
\begin{equation*}
0 \to H^0_b(\Gamma; Z^{j-1}I^\bullet) \to H^0_b(\Gamma; I(j)) \to H^0_b(\Gamma; Z^{j}I^\bullet) \to H^1_b(\Gamma; Z^{j-1}I^\bullet) \to 0
\end{equation*}
for $0 \leq j \leq n-1$. Hence, for every $1 \leq i \leq n$, we obtain the following identification (using Theorem~\ref{thm:basic:properties:bdd:coho}(1)):
\begin{equation*}
H^i_b(\Gamma; V) \cong H^1_b(\Gamma; Z^{i-2}I^{\bullet}) \cong \mathrm{coker}(I(i-1)^{\Gamma} \to (Z^{i-1}I^{\bullet})^{\Gamma}).
\end{equation*}
Using the left exactness of $(-)^{\Gamma}$, the latter is then isomorphic to $H^i(I^{\bullet, \Gamma})$. This finishes the proof for $1 \leq i \leq n$.

Let us consider the case $i = n + 1 \geq 2$. For $0 \leq j \leq n-1$, we have a long exact sequence (Theorem~\ref{thm:basic:properties:bdd:coho}(3)) 
\begin{equation*}
\cdots \to 0 \to H^{n-j}_b(\Gamma; Z^{j}I^{\bullet}) \xrightarrow{\tau} H^{n-j+1}_b(\Gamma; Z^{j-1}I^{\bullet}) \to H^{n-j+1}_b(\Gamma; I(j)) \to \cdots 
\end{equation*}
This shows that there exists a chain of injective maps
$$H_b^{n+1}(\Gamma; V) \supset H^{n}_b(\Gamma; Z^0 I^{\bullet}) \supset H^{n-1}_b(\Gamma; Z^1 I^{\bullet}) \supset \cdots \supset H^1_b(\Gamma; Z^{n-1}I^{\bullet}).$$ 
Moreover, since the last object fits into the long exact sequence
$$H_b^0(\Gamma; I(n)) = I(n)^{\Gamma} \to H_b^0(\Gamma; Z^{n}I^{\bullet}) = Z^{n}I^{\bullet}{}^{\Gamma} \to H^1_b(\Gamma; Z^{n-1}I^{\bullet}) \to H^1_b(\Gamma; I(n))$$
we have
$$H^{n+1}(I^{\bullet, \Gamma}) \cong \mathrm{coker}(I(n)^{\Gamma} \to (Z^{n}I^{\bullet})^{\Gamma}) \subset H^1_b(\Gamma; Z^{n-1}I^{\bullet}),$$
whence the desired inclusion $H^{n+1}(I^{\bullet, \Gamma}) \hookrightarrow H^{n+1}_b(\Gamma; V)$ follows. 
\end{proof}

\begin{rem}\label{rem:acyclic_res}
Assuming that the resolution $I(\bullet)$ in Proposition \ref{prop:acyclic_res} is strong \cite[Section~7.1]{monod}, we may also identify the maps in Proposition \ref{prop:acyclic_res} directly using the standard resolution of $V$. Let 
$$0 \to V = C_b^{-1}(\Gamma; V) \to C_b^0(\Gamma; V) \xrightarrow{\delta^0} C_b^1(\Gamma; V) \xrightarrow{\delta^1} \cdots$$
be the standard resolution of $V$ by relatively injective Banach $\Gamma$-modules as described earlier. 
There is an essentially unique 
morphism of resolutions $(g^\bullet \colon I(\bullet) \to C_b^{\bullet}(\Gamma; V))$ of $V$ (see, for example, \cite[Section~7.2]{monod})
$$
\xymatrix{
0 \ar[r] & V \ar@{=}[d] \ar[r]^{\epsilon} & I(0) \ar[d]^{g^0} \ar[r]^{\delta^0} & I(1) \ar[d]^{g^1} \ar[r]^{\delta^1} & \cdots \\
0 \ar[r] & V \ar[r]^(0.4){\epsilon} & C_b^0(\Gamma; V) \ar[r]^{\delta^0} & C_b^1(\Gamma; V) \ar[r]^(0.6){\delta^1} & \cdots 
}
$$
The maps in Proposition \ref{prop:acyclic_res} are identified with the maps in cohomology that are induced by the morphism $g^{\bullet} \colon \big(I(\bullet), \delta^{\bullet}\big) \to \big(C_b^{\bullet}(\Gamma; V), \delta^{\bullet}\big)$, after first passing to 
the $\Gamma$-invariants. To see this,  let $Z^{j-1}C_b^{\bullet}$ denote the image of 
$\delta^{j-1} \colon C_b^{j-1}(\Gamma; V) \to C_b^j(\Gamma; V),$ and
consider the associated mophisms of short exact sequences:
$$
\xymatrix{
0 \ar[r] & Z^{j-1}I^\bullet \ar[d]^-{g^j} \ar[r] & I(j) \ar[r] \ar[d]^-{g^j} & Z^{j}I^{\bullet} \ar[d]^-{g^{j+1}} \ar[r] & 0 \\
0 \ar[r] & Z^{j-1}C_b^\bullet \ar[r] & C_b^{j}(\Gamma; V) \ar[r] & Z^{j}C_b^{\bullet} \ar[r] & 0.
}
$$
Then, applying the method of Proposition \ref{prop:acyclic_res} to the standard resolution of $V$ and using the functoriality of the long exact sequences in Theorem \ref{thm:basic:properties:bdd:coho}(3), we obtain the following identifications for $0 \leq i \leq n$:
$$
\xymatrix{
H^i(I(\bullet)^{\Gamma}, \delta^{\bullet}) \cong \mathrm{coker}\big(I(i-1)^{\Gamma} \to (Z^{i-1}I^{\bullet})^{\Gamma}\big) \ar@<-14ex>[d]^-{H^i(g^\bullet)} \ar@<8ex>[d]  \ar[r]^(0.7){\cong} & H^i_b(\Gamma; Z^{-1} I^{\bullet}) \ar@{=}[d] \\
H^i(C^{\bullet}_b(\Gamma; V)^{\Gamma}, \delta^{\bullet}) \cong \mathrm{coker}\big(C^{i-1}_b(\Gamma; V)^{\Gamma} \to (Z^{i-1}C^{\bullet}_b)^{\Gamma}\big)  \ar[r]^(0.75){\cong} & H^i_b(\Gamma; Z^{-1}C^{\bullet}_b) 
}
$$
where the vertical maps are induced by $g^{\bullet}$ and $Z^{-1}I^\bullet = Z^{-1}C^\bullet_b = V$. The case $i=n+1$ is similar. 
It follows from the properties of $g^{\bullet}$ that $H^i(g^\bullet)$ is norm non-increasing (see~\cite[Theorem~4.16]{Frigerio:book}, \cite{monod}). 
\end{rem}
 
 
 
 \subsection{Boundedly acyclic groups and spaces} We consider the following class of (discrete) groups and topological spaces.
For any group $\Gamma$, we recall that $\R$ is viewed as a Banach $\Gamma$-module with the trivial action.
 
\begin{defi}\label{defi:intro:VBCR}
Let $n \geq 1$ be an integer or $n=\infty$. 
\begin{itemize}
\item[(1)] A topological space $X$ is called \emph{boundedly $n$-acyclic} if $H^0_b(X; \R) \cong \R$ and $H^i_b(X; \R) = 0$ for $1 \leq i \leq n$. We say that $X$ is \emph{boundedly acyclic} if $X$ is boundedly $\infty$-acyclic. 

\item[(2)] A group $\Gamma$ is called \emph{boundedly $n$-acyclic} if $H^i_b(\Gamma; \R) = 0$ for $1 \leq i \leq n$. We say that $\Gamma$ is a \emph{boundedly acyclic group} if $\Gamma$ is boundedly $\infty$-acyclic. 

\item[(3)] We denote by $\vbcr_n$ the class of \emph{boundedly $n$-acyclic groups} when $n < \infty$. For $n= \infty$, we denote the class of boundedly acyclic groups by $\vbcr$.
\end{itemize}
\end{defi}

\begin{rem}
Note that since every group $\Gamma$ has $H_b^0(\Gamma; \R) \cong \R$, we omitted this condition from the definition of a boundedly $n$-acyclic group.
\end{rem}

The family of boundedly acyclic groups was introduced by L\"oh~\cite{Loeh:dim} under the name of \emph{groups with bounded cohomological dimension zero}. Our choice of terminology follows the definition of an \emph{acyclic space} in homotopy theory (i.e., a space $X$ with trivial reduced singular 
homology in all degrees). As a consequence of the \emph{Mapping Theorem}, note that a path-connected space $X$ is boundedly $n$-acyclic if and only if $X$ has boundedly $n$-acyclic fundamental group.  

 \begin{example} \label{example:vbcr}
Amenable groups are boundedly acyclic. Moreover, the group of homeomorphisms of $\R^n$ with compact support was shown to be boundedly acyclic by  Matsumoto--Morita~\cite{Matsu-Mor}. Later, following the methods of \cite{Matsu-Mor}, L\"oh showed that the class of mitotic groups is also contained in $\vbcr$~\cite[Theorem~1.2]{Loeh:dim}. 

Since there exist mitotic groups which contain non-abelian free subgroups (see~\cite{Loeh:dim}), it follows that $\vbcr$ is not closed under taking subgroups in general. In fact, a much more general result has been shown recently: every (finitely generated) group embeds into a (finitely generated) boundedly acyclic group~\cite{fflm1, monod:thompson}. In addition, the results by Matsumoto--Morita \cite{Matsu-Mor} and L\"oh \cite{Loeh:dim} were generalized to a proof that binate groups lie in $\vbcr$~\cite{fflm1}. Finally, more examples of  homeomorphism and diffeomorphism groups which lie in $\vbcr$ have been found by Monod--Nariman~\cite{monodnariman}.

Recently Monod showed that lamplighter groups are boundedly acyclic~\cite{monod:thompson},
providing a new class of non-amenable finitely generated boundedly acyclic groups. 
Based on this result, Monod also showed that Thompson's group $F$ is boundedly acylic (it is a well-known open question whether $F$ is in fact amenable or not). 
\end{example}

\begin{example}\label{example:vbcr:n}
There are many known examples of $2$-boundedly acyclic groups, 
most notably lattices in higher rank Lie groups (both in the cocompact case~\cite{Burger_Monod_99}
and in the general case~\cite{Burger_Monod_2002}).

Other examples of $2$-boundedly acyclic groups arise from dynamics. Most of these examples have been discussed in the recent work by Fournier-Facio 
and Lodha~\cite{fflodha} in connection with the notion of \emph{commuting conjugates}. However, many $2$-boundedly acyclic groups coming from dynamics were known much earlier. We mention here Thompson's group $F$~\cite{cfp_96} (now known to be boundedly acyclic~\cite{monod:thompson}), 
Thompson's group $V$~\cite{fflm2} (now known to be boundedly acyclic~\cite{konstantin}),
and some homeomorphism groups~\cite{bowden}.

Examples of $3$-boundedly acyclic groups can be found again among certain lattices 
in higher rank Lie groups~\cite{monod07, monod10}, e.g., lattices in $\textup{SL}_n$ are known to be $3$-boundedly acyclic~\cite{monod04}.
Moreover, Bucher and Monod also showed that Burger--Mozes groups
are $3$-boundedly acyclic~\cite{Bucher-Monod} (extending the classical result about their $2$-bounded acyclicity~\cite{Burger_Monod_99}).


There are further examples of non-amenable groups whose bounded cohomology vanishes in low degrees that have been found recently~\cite{fflm2, fflm1, monodnariman}.
\end{example}
 
\begin{prop} \label{prop:trivial_mod} 
Let $\Gamma \in \vbcr_n$. Then $H^i_b(\Gamma; V) = 0$ for every $\R$-generated trivial Banach $\Gamma$-module $V$ and $1 \leq i \leq n$.
\end{prop}
\begin{proof}
We need to show that the cochain complex of $\Gamma$-invariants associated to the standard resolution $(V, C^\bullet_b(\Gamma; V), \delta^\bullet)$ of $V$
\begin{equation*}\label{eq:res:inv:standard}
0 \to C_b^0(\Gamma; V)^\Gamma \xrightarrow{\delta^0} C_b^1(\Gamma; V)^\Gamma \xrightarrow{\delta^1} C_b^2(\Gamma; V)^\Gamma \xrightarrow{\delta^2} \cdots 
\end{equation*}
has vanishing cohomology in degrees $1 \leq i \leq n$ when $V = \ell^{\infty}(S, \R)$ is an $\R$-generated trivial Banach $\Gamma$-module.
By assumption, this holds for $V = \R$. For the general case, let $V = \ell^{\infty}(S, \R)$ be an $\R$-generated trivial Banach $\Gamma$-module.
Then we have a natural identification of Banach $\Gamma$-modules
$$
C^\bullet_b(\Gamma; V) \cong \ell^{\infty}(S, C_b^\bullet(\Gamma; \mathbb{R}))
$$
induced by
$$
f \mapsto \Big( s \mapsto \big((g_0, \cdots, g_\bullet) \mapsto f(g_0, \cdots, g_\bullet)(s) \big)  \Big).
$$
Using that the $\Gamma$-action on $S$ is trivial, these identifications yield an identification of 
cochain complexes
$$
\ell^{\infty}(S, C_b^\bullet(\Gamma; \mathbb{R}))^\Gamma \cong \ell^{\infty}(S, C_b^\bullet(\Gamma; \mathbb{R})^\Gamma).
$$
The $n$-truncation of $C^\bullet_b(\Gamma; \R)^\Gamma$
$$
0 \to C^0_b(\Gamma; \R)^\Gamma \to \cdots \to C^{n-1}_b(\Gamma; \R)^\Gamma \to \ker(\delta^{n}) \to 0 \to \cdots 
$$
is an acyclic cochain complex of Banach spaces. Then the result follows because the functor $\ell^{\infty}(S, -)$ preserves short exact sequences of Banach spaces, using the arguments of \cite[Lemma 8.2.4]{monod}, therefore, it also preserves acyclic cochain complexes of Banach spaces.  
\end{proof}

\begin{rem}
Proposition \ref{prop:trivial_mod} shows that the definitions of boundedly $n$-acyclic homomorphism and boundedly $n$-acyclic group are compatible in the following sense: a group $\Gamma$ is boundedly $n$-acyclic if and only if the trivial homomorphism $\Gamma \to 1$ is boundedly $n$-acyclic. The analogous statement also holds for topological spaces.
\end{rem}

\begin{rem}  \label{W-bdd-acyclic2}
The previous result explains the motivation for introducing the class of $\R$-generated Banach modules. 
More precisely, the class of $\R$-generated Banach $\Gamma$-modules is a natural choice of a family of Banach $\Gamma$-modules for which boundedly acyclic groups $\Gamma$ have vanishing bounded cohomology. 
However, we do not know if the vanishing result of Proposition \ref{prop:trivial_mod} is true for general \emph{dual normed} trivial  $\R[\Gamma]$-modules $V$. On the other hand, given a Banach space $W$, the proof of Proposition \ref{prop:trivial_mod} shows more generally that the homomorphism $\phi \colon\Gamma \to 1$ is $\langle W \rangle$-boundedly $n$-acyclic (see Remark \ref{W-bdd-acyclic}) if and only if $H^i_b(\Gamma; \phi^{-1}W) = 0$ for $1 \leq i \leq n$. 
\end{rem}

\section{Proof of Theorems~\ref{main:thm:intro:amenable} and~\ref{main:thm:intro:finite}}\label{section:proof:amenable:maps}

\subsection{Theorem~\ref{main:thm:intro:amenable} for discrete groups} In this subsection, we will state and prove the analogue of Theorem \ref{main:thm:intro:amenable} in the case of maps $f \colon X \to Y$ which arise from homomorphisms $\phi \colon \Gamma \to K$ of discrete 
groups. 

Let $H$ be a subgroup of a discrete group $\Gamma$. We begin by recalling the definition of a \emph{left} $H$-\emph{invariant mean} on $\ell^\infty(\Gamma)$.

\begin{defi}
A \emph{left} $H$-\emph{invariant mean} on $\ell^\infty(\Gamma)$ is a linear functional $$m \colon \ell^\infty(\Gamma) \to \R$$
such that 
\begin{enumerate}
\item[(a)] $\inf_{g \in \, \Gamma} f(g) \leq m(f) \leq \sup_{g \in \, \Gamma} f(g)$ for every $f \in \, \ell^\infty(\Gamma)$;

\item[(b)] $m(h \cdot f) = m(f)$ for every $f \in \, \ell^\infty(\Gamma)$ and $h \in \, H$.
\end{enumerate}
Here $H$ acts on $\ell^\infty(\Gamma)$ by the restriction of the standard $\Gamma$-action:
$$
g \cdot f ( g_0) = f(g^{-1} g_0)
$$
for $g, g_0 \in \, \Gamma$ and $f \in \, \ell^\infty(\Gamma)$. 
\end{defi}

\begin{rem}
When $H = \Gamma$ the previous definition leads to the classical definition of amenability,
i.e., a group $\Gamma$ is amenable if it admits a left $\Gamma$-invariant mean.
\end{rem}

The existence of a left $H$-invariant mean on $\ell^\infty(\Gamma)$ is equivalent to the existence of a non-trivial $H$-invariant (linear) functional 
$\xi \in \ell^\infty(\Gamma)'$, where the latter denotes the \emph{topological} dual of $\ell^\infty(\Gamma)$ (cf. \cite[Lemma~3.2]{Frigerio:book}).
More precisely, we have the following:

\begin{prop}\label{lemma:invariant:continuous:linear:functional:same:invariant:mean}
Let $H$ be a subgroup of a discrete group $\Gamma$. Then the following are equivalent:
\begin{enumerate}
\item[(1)] $H$ is amenable.

\item[(2)] There exists a left $H$-invariant mean on $\ell^\infty(\Gamma)$.

\item[(3)] There exists a non-trivial $H$-invariant (continuous) functional $\xi \in \ell^\infty(\Gamma)'$.
\end{enumerate}
\end{prop}
\begin{proof}
The equivalence $(1) \Leftrightarrow (2)$ is a classical result on the amenability of subgroups (see, for example,~\cite[Theorem~6]{Caprace_Monod}).
$(2) \Rightarrow (3)$ is obvious: simply set $\xi (f) := m(f)$, where $m \colon \ell^{\infty}(\Gamma) \to \R$ is a left $H$-invariant mean. The implication $(3) \Rightarrow (2)$ is shown \emph{verbatim} as in the classical absolute case, which can be found in~\cite{Frigerio:book}; indeed, the only difference in the present relative situation is working with $H$-invariance instead of $\Gamma$-invariance as in the absolute case.
\end{proof}

We are now ready to state and prove the version of Theorem~\ref{main:thm:intro:amenable} for discrete groups.
We refer the reader to Definition~\ref{def:amenable:maps} for the notion of amenable homomorphism.

\begin{thm}\label{thm:main:groups:amenable}
Let $\phi \colon \Gamma \to K$ be a surjective homomorphism of discrete groups and let $H$ denote the kernel of $\phi$. 
Then the following are equivalent:
\begin{enumerate}
\item[(1)] $\phi$ is an amenable homomorphism.
\item[(2)] For all dual normed $\mathbb{R}[\Gamma]$-modules $V$ the induced inflation map
$$
H^1_b(\phi ; \textup{I}_V) \colon H_b^1(K; V^H) \to H_b^1(\Gamma; V)
$$
is an isomorphism.
\item[(3)] $H$ is amenable.
\end{enumerate}
\end{thm}
\begin{proof}
We will prove the implications
$$(3) \Rightarrow (1)\Rightarrow (2) \Rightarrow (3).$$

\medskip

\noindent \emph{Proof of} $(3) \Rightarrow (1)$. This is an algebraic version of the \emph{Mapping Theorem} with twisted coefficients given by dual normed $\R[\Gamma]$-modules (see, for example,~\cite[Corollary~7.5.10]{monod}). 

\medskip 

\noindent \emph{Proof of} $(1) \Rightarrow (2)$. This is obvious. 

\medskip 

\noindent \emph{Proof of} $(2) \Rightarrow (3)$. Our proof follows the classical approach to characterize amenability using the Johnson class in bounded cohomology~\cite{Johnson}, \cite[3.4]{Frigerio:book}. Recall that the Johnson class is constructed as follows. Let 
$(\ell^\infty(\Gamma) \slash \R)'$ denote the topological dual of the quotient space $\ell^\infty(\Gamma) \slash \R$, where we identify
$\R$ with the subspace of constant functions. 
Note that the space $(\ell^\infty(\Gamma) \slash \R)'$ can be identified with the subspace of functionals in $\ell^\infty(\Gamma)'$ which vanish on the constant functions. We consider the bounded $1$-cocycle
$$
J \in \, C_b^1(\Gamma; (\ell^\infty(\Gamma) \slash \R)')
$$
defined by $J(g_0, g_1) = \delta_{g_1} - \delta_{g_0}$. 
Here, for every $g \in \Gamma$, $\delta_g \in \ell^\infty(\Gamma)'$ denotes the Dirac functional at $g$:
$f \mapsto \delta_g(f) := f(g)$.
An easy computation~\cite[Section~3.4]{Frigerio:book} shows that $J$ is also $\Gamma$-invariant and hence it defines a class in bounded cohomology
$$
[J] \in H_b^1(\Gamma; (\ell^\infty(\Gamma) \slash \R)')
$$
called \emph{the Johnson class}. We are interested in the image of the Johnson class under the restriction map 
$$\res^1 \colon H^1_b(\Gamma;  (\ell^\infty(\Gamma) \slash \R)') \to H^1_b(H;  (\ell^\infty(\Gamma) \slash \R)').$$
By the definition of the restriction map, an explicit representative of $\res^1[J]$ is given by 
$$
\Res^1(J)(h_0, h_1) \coloneqq \delta_{h_1} - \delta_{h_0} \in \, C_b^1(H; (\ell^\infty(\Gamma) \slash \R)')^H.
$$
On the other hand, using assumption (2) and Corollary~\ref{rem:1:amenable:1:restriction}, it follows that $\res^1[J] = 0$. 
This means that there exists 
$$
\alpha \in C_b^0(H; (\ell^\infty(\Gamma) \slash \R)')^H
$$
such that $\delta \alpha = \res^1(J)$. Using the construction of $J$, we are going to show that there exists a non-trivial $H$-invariant functional 
$\xi \in \ell^\infty(\Gamma)'$. This will imply the required result in (3) by applying Proposition~\ref{lemma:invariant:continuous:linear:functional:same:invariant:mean}. For each $h\in H$, we define $\hat{\alpha}(h) \in \, \ell^\infty(\Gamma)'$ to be the 
(continuous) functional given by
$$
\hat{\alpha}(h)(f) = \alpha(h)(\pi(f))
$$
where $\pi \colon \ell^\infty(\Gamma) \to \ell^\infty(\Gamma)/\R$ denotes the $\Gamma$-equivariant quotient map. This allows us to define our candidate $\xi \in \ell^\infty(\Gamma)'$ by
$$
\xi \coloneqq \delta_1 - \hat{\alpha}(1).
$$
First, $\xi \in \ell^{\infty}(\Gamma)'$ is non-trivial: for every non-zero constant function $f$, we have $\hat{\alpha}(1)(f) = 0$ and $\delta_1(f) \neq 0$.
Hence, it remains to prove that $\xi$ is $H$-invariant. To this end, we note first that 
$\hat{\alpha}$ defines an element in $C_b^0(H; \ell^\infty(\Gamma)')^H$: for every $h, h' \in H$ and $f \in \, \ell^\infty(\Gamma)$,
\begin{align*}
(h' \cdot \hat{\alpha})(h)(f) &= \hat{\alpha}(h'^{-1} h) (h'^{-1} f) = \alpha(h'^{-1} h) (\pi(h'^{-1}f)) =  \alpha(h'^{-1} h) (h'^{-1} \pi(f)) \\
&=(h' \cdot \alpha)(h) (\pi(f)) = \alpha(h) (\pi(f)) = \hat{\alpha}(h) (f).
\end{align*}
Note that we used the fact that $\alpha$ is $H$-invariant.
Moreover, since $\delta \alpha = \res^1(J)$, we also have
$$
\delta_{h_1} - \delta_{h_0} = \hat{\alpha}(h_1) - \hat{\alpha}(h_0)
$$
for all $h_0, h_1 \in \, H$. In particular, the following equality holds for all $h \in H$, 
$$
\delta_h - \hat{\alpha}(h) = \delta_1 - \hat{\alpha}(1).
$$
These computations imply that for all $h \in \, H$, we have
\begin{align*}
h \cdot \xi &= h \cdot (\delta_1 - \hat{\alpha}(1)) = h \cdot \delta_1 - h \cdot \hat{\alpha}(1) \\
& = \delta_h - h \cdot \hat{\alpha}(h^{-1}h) \\
&= \delta_h - (h \cdot \hat{\alpha})(h) = \delta_h - \hat{\alpha}(h)  \\
&=\delta_1 - \hat{\alpha}(1) = \xi.
\end{align*}
This completes the proof of the existence of a non-trivial $H$-invariant (continuous) functional $\xi$ on $\ell^\infty(\Gamma)'$.
By Proposition~\ref{lemma:invariant:continuous:linear:functional:same:invariant:mean}, this implies that $H$ is amenable, as required.
\end{proof}

\subsection{Proof of Theorem~\ref{main:thm:intro:amenable}} We recall the statement (see Definition~\ref{def:amenable:maps}):

\begin{intro_thmA}
Let $f \colon X \to Y$ be a map of based path-connected spaces, let $f_* \colon \pi_1(X) \to \pi_1(Y)$ be the induced homomorphism between the fundamental groups and let $H$ denote its kernel. Let $F$ denote the homotopy fiber of $f$ and suppose that $F$ is path-connected (equivalently, $f_*$ is surjective). Then the following are equivalent:
\begin{enumerate}
\item[(1)] $f$ is an amenable map.
\item[(2)] For all dual normed $\mathbb{R}[\pi_1(X)]$-modules $V$, the induced inflation map
$$
H^1_b(f ; \textup{I}_V) \colon H_b^1(Y; V^H) \to H_b^1(X; V)
$$
is an isomorphism.
\item[(3)] $F$ is amenable. 
\end{enumerate}
\end{intro_thmA}
\begin{proof}
By Proposition \ref{prop:compare_amenable}, conditions (1) and (2) of Theorem \ref{main:thm:intro:amenable} are equivalent respectively to the conditions (1) and (2) of Theorem \ref{thm:main:groups:amenable} for $f_*$. Moreover, the long exact sequence of homotopy groups 
$$\cdots \to \pi_2(Y) \to \pi_1(F) \to \pi_1(X) \xrightarrow{f_*} \cdots$$
shows that the kernel of the surjective homomorphism $\pi_1(F) \to \mathrm{ker}(f_*)$ is abelian, therefore, also amenable. 
As already discussed in Remark~\ref{rem:amenable_space}, condition (3) is equivalent to the condition that $\pi_1(F)$ is amenable. It follows that the condition (3) of Theorem \ref{main:thm:intro:amenable} is equivalent to the condition of Theorem \ref{thm:main:groups:amenable}(3) for $f_*$. Then the result follows directly from Theorem \ref{thm:main:groups:amenable}. 
\end{proof}

\subsection{Proof of Theorem \ref{main:thm:intro:finite}} In the case of arbitrary coefficients, the analogue of Theorem \ref{thm:main:groups:amenable} leads to a characterization of surjective homomorphisms whose kernel is finite. We recall 
the statement of Theorem \ref{main:thm:intro:finite}:

\begin{intro_thmB}
Let $\phi \colon \Gamma \to K$ be a surjective homomorphism of discrete groups and let $H$ denote the kernel of $\phi$. 
Then the following are equivalent:
\begin{enumerate}
\item[(1)] For all Banach $\Gamma$-modules $V$, the induced inflation map
$$
H^{\bullet}_b(\phi ; \textup{I}_V) \colon H_b^{\bullet}(K; V^H) \to H_b^{\bullet}(\Gamma; V)
$$
is an isometric isomorphism. 
\item[(2)] For all Banach $\Gamma$-modules $V$, the induced inflation map
$$
H^1_b(\phi ; \textup{I}_V) \colon H_b^1(K; V^H) \to H_b^1(\Gamma; V)
$$
is an isomorphism.
\item[(3)] $H$ is finite.
\end{enumerate}
\end{intro_thmB}
\begin{proof}
We will prove the implications
$$(3) \Rightarrow (1)\Rightarrow (2) \Rightarrow (3).$$

\medskip

\noindent \emph{Proof of} $(3) \Rightarrow (1)$. Since finite subgroups are compact, this readily follows from the corresponding statement for topological groups in~\cite[Proposition~8.5.6]{monod}. 

\medskip 

\noindent \emph{Proof of} $(1) \Rightarrow (2)$. This is obvious. 

\medskip 

\noindent \emph{Proof of} $(2) \Rightarrow (3)$. Our proof follows Frigerio's characterization of finite groups~\cite[Section~3.5]{Frigerio:book}.
Let $\ell^1(\Gamma)$ be the Banach $\Gamma$-module of summable real functions on $\Gamma$ with countable support. We restrict our attention to the kernel of the summation function
$$
\sigma \colon \ell^1(\Gamma) \to \R, \quad \quad f \mapsto \sum_{g \in \, \Gamma} f(g),
$$
which is a Banach $\Gamma$-submodule of $\ell^1(\Gamma)$, denoted by $\ell^1_0(\Gamma)$, and consider the analogous version of the Johnson class in this setting. More precisely, let $J \in \, C_b^1(\Gamma; \ell^1_0(\Gamma))$ be the cocycle 
defined by
$$J(g_0, g_1) = \delta_{g_1} - \delta_{g_0}, \quad \quad \mbox{for all } g_0, g_1 \in \, \Gamma,$$
where $\delta_g \in \, \ell^1(\Gamma)$ denotes the characteristic function for $\{g\} \subset \Gamma$. Since $J$ is $\Gamma$-invariant (see~\cite[Section~3.5]{Frigerio:book}), it defines a class in bounded cohomology $[J] \in \, H^1_b(\Gamma; \ell^1_0(\Gamma))$.
As in the proof of Theorem~\ref{thm:main:groups:amenable}, we are going to study the image of the class 
$[J]$ under the restriction map 
$$\res^1 \colon H^1_b(\Gamma;  \ell_0^1(\Gamma)) \to H^1_b(H; \ell_0^1(\Gamma)).$$
By assumption (2) and Corollary~\ref{rem:1:amenable:1:restriction} we know that $\res^1[J] = 0$, so there exists 
$$
\alpha \in C_b^0(H; \ell^1_0(\Gamma))^H
$$
such that 
\begin{equation}\label{eq:finite:group:proof} \tag{$\dagger$}
\alpha (h_1) - \alpha(h_0) = \delta_{h_1} - \delta_{h_0}, \quad \quad \mbox{for all }  h_0, h_1 \in \, H.
\end{equation}
We can then use $\alpha$ to define a non-trivial $H$-invariant summable function $$\xi \coloneqq \delta_1 - \alpha(1) \in \, \ell^1(\Gamma).$$
Note that $\xi$ is non-trivial since $\sigma(\xi) = \sigma(\delta_1) - \sigma(\alpha(1)) = 1$. Moreover, for every $h \in \, H$, we have
$$
h \cdot \xi = h \cdot (\delta_1 - \alpha(1)) = h \cdot \delta_1 - h \cdot \alpha(1) = \delta_h - \alpha(h) = \delta_1 - \alpha(1) = \xi,
$$
where we used the $H$-invariance of $\alpha$ and the formula~\eqref{eq:finite:group:proof}. This shows that $\xi \in \, \ell^1(\Gamma)^H$.

Since $\xi$ is $H$-invariant, it must be constant on each coset of $H$ in $\Gamma$. Combining this fact with the non-triviality of 
$\xi \in \, \ell^1(\Gamma)$, it follows that the group $H$ must be finite.
\end{proof}

\begin{rem}\label{rem:finite:groups:spaces}
 It would be interesting to find a suitable generalization of Theorem \ref{main:thm:intro:finite} in the context of topological spaces. Note that the argument that we used 
to deduce Theorem \ref{main:thm:intro:amenable} from Theorem \ref{thm:main:groups:amenable} cannot be applied in the same way for two related reasons. First, given a map $f \colon X \to Y$ (as in Theorem \ref{main:thm:intro:amenable}), there is an exact sequence of groups 
$$\cdots \to \pi_2(Y) \to \pi_1(F) \to \ker(f_*) \to 1;$$
therefore, whether the homomorphism $\pi_1(F) \to \ker(f_*)$ satisfies the conditions of Theorem \ref{main:thm:intro:finite} depends also on the homomorphism $\pi_2(Y) \to \pi_1(F)$. Second, it is not known whether the inflation maps associated with the canonical map $c_X \colon X \to B \pi_1(X)$ are isomorphisms for arbitrary Banach $\pi_1(X)$-modules (cf. Proposition \ref{prop:compare_amenable}). In fact, it seems more likely that this would only hold when the higher homotopy groups of $X$ are finite. Similarly, it seems likely that the analogous characterization to Theorem B for topological spaces would characterize the maps $f \colon X \to Y$ whose homotopy fiber $F$ has \emph{finite} homotopy groups. In future work, we aim to address these questions about coefficients in general Banach $\pi_1$-modules and explore possible connections with properties of the higher homotopy groups.
\end{rem}

\section{Proof of Theorem~\ref{main:thm:intro}}\label{sec:proofs}

\subsection{Theorem~\ref{main:thm:intro} for discrete groups} We will first prove the following algebraic version of Theorem \ref{main:thm:intro} for discrete groups
(see Definition~\ref{def:intro:bdd:cohom:equiv}).


\begin{thm}\label{thm:main:groups}
Let $\phi \colon \Gamma \to K$ be a homomorphism of discrete groups and let $H$ denote its kernel. Let $n \geq 0$ be an integer or $n=\infty$. Then the following are equivalent:
\begin{enumerate}
\item[(1)] $\phi$ is boundedly $n$-acyclic.
\item[(2)] The induced restriction map
$$H^i_b(\phi; V) \colon H_b^i(K; V) \to H_b^i(\Gamma; \phi^{-1}V)$$
is surjective for $0 \leq i \leq n$ and every $\R$-generated Banach K-module $V$.
\item[(3)] $\phi$ is surjective and $H^i_b(\Gamma; \phi^{-1}V) = 0$ for $1 \leq i \leq n$ and every relatively injective $\R$-generated Banach $K$-module $V$.
\item[(4)]  $\phi$ is surjective and $H$ is a boundedly $n$-acyclic group.
\end{enumerate}
\end{thm}
\begin{proof}
We will prove the implications 
$$(1) \Rightarrow (2) \Rightarrow (3) \Rightarrow (4) \Rightarrow (1).$$

\medskip

\noindent \emph{Proof of} $(1) \Rightarrow (2)$. This is obvious. 

\medskip 

\noindent \emph{Proof of} $(2) \Rightarrow (3)$. Suppose that $\phi$ induces a surjective map  
$$H^i_b(\phi; V) \colon H_b^i(K; V) \to H_b^i(\Gamma; \phi^{-1}V)$$
for $0 \leq i \leq n$ and every $\R$-generated Banach $K$-module $V$. Let $V$ be a \emph{relatively injective} $\R$-generated Banach $K$-module. 
Then $H^0_b(K; V) \cong V^K$ and $H_b^i(K; V) = 0$ for all $i \geq 1$~(see~Theorem~\ref{thm:basic:properties:bdd:coho}(2)). 
Therefore,
$$
H_b^i(\Gamma; \phi^{-1} V) = 0
$$
for $1 \leq i \leq n$. It remains to prove that $\phi$ is surjective. Let $K' \coloneqq \mathrm{im}(\phi)$ denote the image of $\phi$. Consider the $K$-set of left cosets $ J \coloneqq K \slash K'$ and the $\R$-generated Banach $K$-module $V\coloneqq \ell^{\infty}(J, \R)$. Then we have (using Theorem~\ref{thm:basic:properties:bdd:coho}(1)):
$$H^0_b(K; V) \cong V^K \cong \R \text{ \  and  \ } H^0_b(\Gamma; \phi^{-1}V) \cong (\phi^{-1}V)^{\Gamma} \cong V^{K'}.$$ 
The map $H^0_b(K; V) \to H^0_b(\Gamma; \phi^{-1}V)$ is surjective (if and) only if $J$ is a transitive $K'$-set. This happens exactly when $J$ is a singleton, that is, only when $\phi$ is surjective. 

\medskip

\noindent \emph{Proof of} $(3) \Rightarrow (4)$. Assuming (3), we have that $\phi$ induces an isomorphism
$$H^i_b(\phi; V) \colon H_b^i(K; V) \to H_b^i(\Gamma; \phi^{-1}V)$$
for every relatively injective $\R$-generated Banach $K$-module $V$ and $0 \leq i \leq n$. To see this, note that $H^i_b(K; V) = 0$ for $i \geq 1$ (by Theorem~\ref{thm:basic:properties:bdd:coho}(2)) and $H^0_b(\phi; V)$ is an isomorphism for every $\R$-generated Banach $K$-module $V$ since $\phi$ is surjective. Now we consider the dual normed $\R[K]$-module
$$V \coloneqq \mathbf{I}_{\{1\}}^K \R = \ell^\infty(K, \R).$$
Viewing $\R$ as a dual normed $\R[K]$-module endowed with the trivial action, we can assume that $V$ is endowed with the usual left action by $K$, as discussed in Remark~\ref{rem:induction:module:action:Monod}. In particular, $\mathbf{I}_{\{1\}}^K \R$ is an $\R$-generated Banach $K$-module.
Moreover, the $\R[K]$-module $V$ is relatively injective because $K$ acts freely on itself~\cite[Definition~4.20 and Lemma~4.22]{Frigerio:book}. Therefore, $\phi$ induces isomorphisms
$$
\phi^* = H^i_b(\phi; V) \colon H^i_b(K; V) \to H^i_b(\Gamma; \phi^{-1} V)
$$
for $0 \leq i \leq n$. Next we consider the diagram of groups
$$
\xymatrix{
H \ar@{}[r]|-*[@]{\subset} \ar[d]_{\phi |_H} & \Gamma \ar[d]^{\phi} \\
\{1\} \ar@{}[r]|-*[@]{\subset}  & K.
}
$$ 
By Proposition~\ref{prop:shapiro2}, Proposition \ref{prop:coeff:isomorphic} and Proposition \ref{prop:shapiro1}, we have a commutative diagram as 
follows
$$
\xymatrix@C=1em{
H^i_b(\{1\}; \R) \ar[r]^{\textbf{i}}_{\cong} \ar[dd]_-{(\phi_{|H})^*} & H^i_b(K; V) \ar[rd]^-{\phi^*}_{\cong} \\
&& H^i_b(\Gamma; \phi^{-1} V) \ar[dl]^-{H^i_b(\textup{id}_\Gamma, \Psi)}_{\cong} \\
H^i_b(H; \R) \ar[r]^(0.36){\textbf{i}}_(.36){\cong} & H^i_b(\Gamma; \ell^\infty(\Gamma, \R)^H) \ ,
}
$$
for every $0 \leq i \leq n$ (where all the coefficients are $\R$-generated Banach modules as noticed in Remark~\ref{rem:all:coeff:shapiro:are:generated}).
This diagram readily implies that $H^i_b(H; \R)$ is isomorphic to $H^i_b(\{1\}; \R)$ for $0 \leq i \leq n$, therefore $H^0_b(H; \R) = \R$ and $H^i_b(H; \R) = 0$ for all $1 \leq i \leq n$, as required. 

\medskip

\noindent \emph{Proof of} $(4) \Rightarrow (1)$.  Let $V$ be an $\R$-generated Banach $K$-module. We are required to show that the induced restriction map
$$
 H^i_b(\phi; V) \colon H^i_b(K; V) \to H^i_b(\Gamma; \phi^{-1} V)
 $$
 is an isomorphism for $0 \leq i \leq n$ and injective for $i = n +1$. Let $D_j$ denote the relatively injective $\R$-generated Banach $K$-module $\ell^{\infty}(K^{j+1}, V)$ and let $W_{j} = \phi^{-1}D_j$ denote the corresponding $\R$-generated Banach $\Gamma$-module (see Remark~\ref{rem:generated:module:examples}). We also write $D_{-1} = V$ and $W_{-1} = \phi^{-1}V$, respectively. 
 
Since $\phi \colon \Gamma \to K$ is surjective, there is a canonical (isometric) isomorphism between the normed modules $C^j_b(K; V)^K$ and $W_j^\Gamma$ for all $j \geq 0$. Therefore it suffices to show that the resolution (i.e. acyclic cochain complex) of Banach $\Gamma$-modules:
\begin{equation} \label{resolution} \tag{*} 
0 \to W_{-1} \to W_0 \to W_1 \to \cdots
\end{equation}
can be used to compute the bounded cohomology of $\Gamma$ with coefficients in $\phi^{-1} V$ in the required range. Since the previous resolution is strong~\cite[Lemma~7.5.5 and Example~2.1.2(i)]{monod}, using Proposition \ref{prop:acyclic_res}
and Remark~\ref{rem:acyclic_res}, it suffices to prove that the Banach $\Gamma$-module $W_j$ satisfies
\begin{equation*}\label{eq:module:zero:acyclicity}
H^i_b(\Gamma; W_j) = 0
\end{equation*}
for $1 \leq i \leq n$ and $j \geq 0$. Moreover, since $H \unlhd \Gamma$ is the kernel of $\phi$, each module $W_j$ is a trivial $\R$-generated Banach $H$-module. (We will use the same 
notation $W_j$ for the corresponding $\R[H]$-module in order to simplify the notation.) Therefore, by Proposition \ref{prop:trivial_mod}, we have $H^i_b(H; W_j) = 0$ for $1 \leq i \leq n$ and $j \geq -1$. Using the Eckmann--Shapiro lemma (Proposition~\ref{prop:shapiro1}), we obtain the following identifications for all $i \geq 0$ and $j \geq -1$:
$$
\mathbf{i} \colon H^i_b(H; W_j) \cong H_b^i(\Gamma; \textbf{I}_H^\Gamma W_j).
$$
Since $H$ acts trivially on $W_j$ and $\phi$ is surjective, we may identify these induction modules as follows (Remark~\ref{rem:induction:module:action:Monod}),
$$\textbf{I}_H^{\Gamma} W_j \cong \ell^\infty(K, \ell^\infty(K^{j+1}, \phi^{-1} V)) \cong \ell^\infty(K^{j+2}, \phi^{-1} V) \cong W_{j+1},$$
which then readily implies
$$
0 = H^i_b(H; W_j) \cong H_b^i(\Gamma; \textbf{I}_H^\Gamma W_j) \cong H_b^i(\Gamma; W_{j+1})
$$
for all $1 \leq i \leq n$ and $j \geq -1$. This finishes the proof of (4) $\Rightarrow$ (1) and completes the proof of the theorem.
\end{proof}

\begin{example}
The following is a well-known open question: Does every surjective homomorphism $\phi \colon \Gamma \to K$ between discrete groups induce injective 
maps $H^\bullet_b(\phi) \colon H^\bullet_b(K; \mathbb{R}) \to H^\bullet_b(\Gamma; \mathbb{R})$ in all degrees? 
Bouarich~\cite{Bouarich:exact} proved that this is always the case in degree $2$ and Huber~\cite{Huber} showed that $H_b^2(\phi)$ is also isometric. 
On the other hand, by applying Theorem~\ref{thm:main:groups} $(4) \Rightarrow (1)$, we conclude that the map 
$$H^2_b(\phi; V) \colon H^2_b(K; V) \to H^2_b(\Gamma; \phi^{-1}V)$$
is injective for every $\R$-generated Banach $K$-module $V$; this is because every discrete group is $1$-acyclic~\cite[Section~2.1]{Frigerio:book}. See 
also \cite[Theorem 12.4.2]{monod}.
\end{example}

\subsection{Applications to boundedly acyclic groups}
We may apply Theorem~\ref{thm:main:groups} to study the hereditary properties of the class $\vbcr_n$ of boundedly $n$-acyclic groups. We recall that this class of groups is not closed under taking subgroups in general (see Example \ref{example:vbcr}). On the other hand, we may characterize the situation for normal subgroups as an immediate consequence of Theorem~\ref{thm:main:groups} $(3) \Leftrightarrow (4)$.

\begin{cor}\label{cor:intro:normal:subgrps:vbcr}
Let $n \geq 1$ be an integer or $n = \infty$. Let $\Gamma$ be a boundedly $n$-acyclic group, let $H \unlhd \Gamma$ be a normal subgroup, and let $\phi \colon \Gamma \to \Gamma \slash H$ be the quotient homomorphism. Then $H$ is a boundedly $n$-acyclic group if and only if 
$$H^{i}_b(\Gamma; \phi^{-1}V) = 0$$ for every relatively injective $\R$-generated Banach $\Gamma \slash H$-module $V$ and $1 \leq i \leq n$.
\end{cor}

Moreover, using Theorem~\ref{thm:main:groups} $(1)\Leftrightarrow(4)$, we also easily deduce the following result about the relationship between boundedly acyclic groups and group extensions. 

\begin{cor}\label{cor:intro:extension:grps:vbcr}
Let $n \geq 1$ be an integer or $n = \infty$. Let $1 \to H \to \Gamma \to K \to 1$ be an extension of discrete groups. 
Suppose that $H \in \, \vbcr_n$. Then, we have that 
$$
\Gamma \in \, \vbcr_n \quad \Longleftrightarrow \quad K \in \, \vbcr_n.
$$
\end{cor}

\begin{rem}
There exists a boundedly acyclic group $\Gamma$ that contains a normal subgroup whose bounded cohomology is infinite dimensional in all degrees $\geq 2$~\cite[Theorem~1.5]{fflm2} and such that the quotient is still boundedly acyclic.
\end{rem}

\begin{example}
Corollary~\ref{cor:intro:extension:grps:vbcr} allows us to construct new discrete groups in $\vbcr_n$ by taking extensions of a group in $\vbcr_n$ by the lattices described in Example~\ref{example:vbcr:n}.
\end{example}

\subsection{Proof of Theorem~\ref{main:thm:intro}} We recall the statement:

\begin{intro_thmC}\label{main:thm} 
Let $f \colon X \to Y$ be a map between based path-connected spaces, let $F$ denote its homotopy fiber, and let $n \geq 0$ be an integer or $n = \infty$.  We denote by $f_* \colon \pi_1(X) \to \pi_1(Y)$ the induced homomorphism between the fundamental groups. Then the following are equivalent:
\begin{enumerate}
\item[(1)] $f$ is boundedly $n$-acyclic.
\item[(2)] The induced restriction map 
$$H^i_b(f; V) \colon H_b^i(Y; V) \to H_b^i(X; f_*^{-1}V)$$ 
is surjective for $0 \leq i \leq n$ and every $\R$-generated Banach $\pi_1(Y)$-module $V$.
\item[(3)] $F$ is path-connected and $H^i_b(X; f_*^{-1}V) = 0$ for $1 \leq i \leq n$ and every relatively injective $\R$-generated Banach $\pi_1(Y)$-module $V$.
\item[(4)]  $F$ is boundedly $n$-acyclic, that is, $H^0_b(F; \mathbb{R}) \cong \mathbb{R}$ and $H^i_b(F; \mathbb{R}) = 0$ for $1 \leq i \leq n$.
\end{enumerate}
\end{intro_thmC}
\begin{proof} The proof of Proposition \ref{prop:compare_bdd_acyclic} (using the \emph{Mapping Theorem}) shows that the conditions (1) and (2) of Theorem \ref{main:thm:intro} are equivalent respectively to the conditions (1) and (2) of Theorem \ref{thm:main:groups} for the homomorphism $f_*$. 

Since $X$ is path-connected by assumption, the homotopy fiber $F$ is path-connected if and only if $f_*$ is surjective. Therefore, the condition (3) of Theorem \ref{main:thm:intro} is also equivalent to the condition of Theorem \ref{thm:main:groups}(3) for the homomorphism $f_*$ -- again, using the \emph{Mapping Theorem}. 

Finally, assuming that $F$ is path-connected ($\Leftrightarrow H^0_b(F; \R) \cong \R$), the canonical map $F \to B\pi_1(F)$ is amenable by the \emph{Mapping Theorem}. Moreover, the long exact sequence of homotopy groups 
$$\cdots \to \pi_2(Y) \to \pi_1(F) \to \pi_1(X) \xrightarrow{f_*} \cdots$$
shows that the kernel of the homomorphism $\pi_1(F) \to \mathrm{ker}(f_*)$ is abelian. As a consequence of the \emph{Mapping Theorem}, the homomorphism $\pi_1(F) \to \mathrm{ker}(f_*)$ is also amenable. In particular, the condition (4) of Theorem \ref{main:thm:intro} is equivalent to the condition of Theorem \ref{thm:main:groups}(4) for $f_*$. Then the result follows directly from Theorem \ref{thm:main:groups}. 
\end{proof}


\begin{scholium}
We recall that a map $f \colon X \to Y$ of path-connected spaces is called \emph{acyclic} if its homotopy fiber $F$ satisfies $\widetilde{H}_*(F; \Z) = 0$. Acyclic maps can be equivalently characterized as the maps $f \colon X \to Y$ which induce 
isomorphisms in cohomology for all local coefficients of abelian groups on $Y$ (see \cite{HH-acyclic, Raptis-acyclic} for more 
details and further characterizations). Thus, the equivalence $(1) \Leftrightarrow (4)$ in Theorem~\ref{main:thm:intro} is an analogue of this classical characterization of acyclic maps. 

We note that the classical characterization of acyclic maps can be shown using the Serre spectral sequence. A version of the Hochschild--Serre spectral sequence for the bounded cohomology of (not necessarily discrete) groups was established and studied by Burger--Monod \cite{Burger_Monod_2002, monod}. There is a subtle point here which concerns the identification of the $E_2$-page; this is due to the fact that bounded cohomology groups (as seminormed spaces) are not Hausdorff in general (see \cite[Chapter 12]{monod}). Still, it is possible to use the Hochschild--Serre spectral sequence directly for the implication $(4) \Rightarrow (1)$ in Theorem~\ref{thm:main:groups} -- our proof is, in fact, essentially not very different (see also~\cite{echtler} for more details). 
We expect that the Hochschild--Serre spectral sequence agrees in general with the hypercohomology spectral sequence associated with the resolution \eqref{resolution} (without any prior assumptions on the $\R[\Gamma]$-modules $W_j$).
\end{scholium}

\subsection{Stability properties of boundedly acyclic maps} We discuss the stability properties of boundedly $n$-acyclic maps under various homotopy-theoretic constructions. First, the characterization of boundedly $n$-acyclic maps in Theorem~\ref{main:thm:intro} has the following immediate consequence. 

\begin{cor}
The class of boundedly $n$-acyclic maps is closed under homotopy pullbacks. More precisely: given a homotopy pullback of based path-connected spaces 
$$
\xymatrix{
X' \ar[r] \ar[d]^{f'} & X \ar[d]^f \\
Y' \ar[r] & Y
}
$$
where $f$ is boundedly $n$-acyclic, then so is $f'$. 

Conversely, if $f'$ is boundedly $n$-acyclic, then so is $f$. 
\end{cor}
\begin{proof}
This follows easily from Theorem~\ref{main:thm:intro} $(1) \Leftrightarrow (4)$ because $f$ and $f'$ have the same homotopy fiber.
\end{proof} 

Furthermore, even though bounded cohomology does not satisfy excision in general, the characterization $(1) \Leftrightarrow (4)$ of Theorem~\ref{main:thm:intro} shows that boundedly ($n$-)acyclic maps satisfy the following \emph{descent} property.

\begin{cor}
The class of boundedly $n$-acyclic maps is closed under glueing of equifibered maps. More precisely: given a diagram of based path-connected spaces 
$$
\xymatrix{
X_1 \ar[d]_{f_1} & X_0 \ar[r] \ar[l] \ar[d]^{f_0} & X_2 \ar[d]^{f_2} \\
Y_1 & Y_0 \ar[r] \ar[l] & Y_2 
}
$$
where both squares are homotopy pullbacks and the three vertical maps are boundedly $n$-acyclic, then the induced map between the homotopy pushouts 
$$f \colon X_1 \bigcup^h_{X_0} X_2 \to Y_1 \bigcup^h_{Y_0} Y_2$$
is also boundedly $n$-acyclic.
\end{cor}
\begin{proof}
The homotopy fiber of $f$ agrees (up to weak homotopy equivalence) with the common homotopy fiber of the maps $f_i, i = 0,1,2$ (this is a classical result which can be found in \cite{Puppe74}; see also \cite{RezkTopos} for a modern 
perspective on homotopical descent). Then the required result follows from Theorem~\ref{main:thm:intro} $(1) \Leftrightarrow (4)$. 
\end{proof}

Concerning the stability of boundedly $n$-acyclic maps with respect to homotopy pushouts, the situation is more delicate:

\begin{cor}
Let 
$$
\xymatrix{
X \ar[r]^g \ar[d]^{f} & X' \ar[d]^{f'} \\
Y \ar[r] & Y'
}
$$
be a homotopy pushout of based path-connected spaces, where $f$ is boundedly $n$-acyclic. Let $H \coloneqq \ker(f_*) \unlhd \pi_1(X)$ and let $H'$ denote the image of $H$ under $g_* \colon \pi_1(X) \to \pi_1(X')$. Then $f'$ is boundedly $n$-acyclic if and only if the normal closure of $H' \leq \pi_1(X')$ is boundedly $n$-acyclic. 
\end{cor}
\begin{proof}
By the van Kampen theorem, we obtain a pushout of groups after applying $\pi_1$ to the given homotopy pushout. Note that $f_*$ is surjective; therefore, so is $f'_*$. By Proposition \ref{prop:compare_bdd_acyclic}, the map $f'$ is boundedly $n$-acyclic if and only if the homomorphism $f'_*$ is boundedly $n$-acyclic. By Theorem~\ref{thm:main:groups}, this holds 
if and only if $\ker(f'_*) \in \vbcr_n$. Using the diagram of pushouts of groups
$$
\xymatrix{
H \ar@{>->}[r] \ar[d] & \pi_1(X) \ar[r]^{g_*} \ar@{->>}[d]^{f_*} & \pi_1(X') \ar@{->>}[d]^{f'_*} \\
\{*\} \ar[r] & \pi_1(Y) \ar[r] & \pi_1(Y')
}
$$
we observe that $\ker(f'_*)$ is the normal closure of $H' \leq \pi_1(X')$ and the result follows. 
\end{proof}

\subsection{An application to simplicial volume}

We end this section with an application of Theorem~\ref{main:thm:intro} to \emph{simplicial volume}. The simplicial volume is a homotopy invariant of oriented compact manifolds introduced by Gromov~\cite{vbc}, which measures the complexity of manifolds in terms of the ``size'' of their fundamental cycles in the singular chain complex. We recall that given an oriented compact manifold (possibly with non-empty boundary), a \emph{(relative) fundamental cycle} $c \in \, C_n(M, \partial M; \R)$
is simply a representative of the \emph{(relative) fundamental class}
$[M, \partial M] \in \, H_n(M, \partial M; \R) \cong \R$.
\begin{defi}
Let $M$ be an oriented compact connected $n$-dimensional manifold with (possibly empty) boundary $\partial M$. The \emph{(relative) simplicial volume} of $M$ is defined to be the following non-negative real number:
$$\sv{M, \partial M} \coloneqq \inf \left\{\sum_{i = 1}^k |\alpha_i| \, \Big| \, \sum_{i = 1}^k \alpha_i \, \sigma_i \mbox{ is a fundamental cycle of $(M, \partial M)$}\right\}.
$$
\end{defi}
In order to compute the simplicial volume of an oriented compact manifold with boundary, it is convenient to extend the definition of bounded cohomology of spaces to the relative case. The relative bounded cochain complex is the subobject of the relative singular cochain complex $C^{\bullet}(M, \partial M; \R)$ which consists of bounded cochains, 
$$
C^\bullet_b(M, \partial M; \R) \coloneqq \left\{\varphi \in \, C^\bullet(M, \partial M; \R) \, \Big| \, \sup_{\sigma \colon \Delta^n \to M} |\varphi(\sigma)| < \infty\right\}. 
$$ 
The cohomology of this cochain complex is the \emph{relative bounded cohomology of $(M, \partial M)$}. There is a natural inclusion of cochain complexes
$$
C^\bullet_b(M, \partial M; \R) \hookrightarrow C^\bullet(M, \partial M; \R)
$$
which induces a map in cohomology
$$
\comp^\bullet_{(M, \partial M)} \colon H^\bullet_b(M, \partial M; \R) \to H^\bullet(M, \partial M; \R)
$$
called the \emph{comparison map}. A useful relationship between the simplicial volume and the comparison map is given by the following elementary fact~\cite{vbc}, \cite[Proposition~5.15]{Loehthesis}:

\begin{prop}[Duality principle]\label{duality:principle}
Let $M$ be an oriented compact connected $n$-manifold with (possibly empty) boundary $\partial M$. 
Then, 
$$
\sv{M, \partial M} > 0 \quad \Longleftrightarrow \comp_{(M, \partial M)}^n \mbox{ is surjective}.
$$
\end{prop}

As an application of Theorem~\ref{main:thm:intro} and Proposition~\ref{duality:principle}, we deduce the following
new result on the vanishing of the relative simplicial volume of compact manifolds with boundary.

\begin{cor}\label{cor:sv:vanishes}
Let $M$ be an oriented compact connected  $n$-manifold with non-empty connected boundary $\partial M$.
Let $i \colon \partial M \to M$ denote the boundary inclusion and let $F$ denote its homotopy fiber. 
Suppose that $F$ is boundedly $(n-1)$-acyclic.
Then, $H_b^k(M, \partial M; \R) = 0$ for all $0 \leq k \leq n$. 
In particular, we have $$\sv{M, \partial M} = 0.$$
\end{cor}
\begin{proof}
Since the homotopy fiber $F$ satisfies the condition $(4)$ of Theorem~\ref{main:thm:intro},
it follows that $i$ is boundedly $(n-1)$-acyclic. In particular, the induced map
$$
H^k_b(i) \colon H^k_b(M; \R) \to H^k_b(\partial M; \R)
$$
is an isomorphism in degrees $0 \leq k \leq n-1$ and injective in degree $k = n$. 
If we now consider the long exact sequence in bounded cohomology of the pair $(M, \partial M)$~\cite[Section~5.7]{Frigerio:book},
$$
\cdots \to H^{k-1}_b(M; \R) \to H^{k-1}_b(\partial M; \R) \to H_b^k(M, \partial M; \R) \to H_b^k(M; \R) \to \cdots
$$
we conclude that $H_b^k(M, \partial M; \R) = 0$ for all $0 \leq k \leq n$. 
Then the vanishing of the relative simplicial volume is an easy application of Proposition~\ref{duality:principle}. 
\end{proof}


 
\bibliographystyle{abbrv}
\bibliography{svbib}

\end{document}